\documentclass[12pt]{amsart}
\usepackage[english]{babel}
\usepackage[utf8]{inputenc}
\usepackage[T1]{fontenc}

\usepackage[margin=1.5in]{geometry}

\usepackage[colorinlistoftodos]{todonotes}
\usepackage[colorlinks=true, allcolors=blue]{hyperref}

\usepackage[alphabetic]{amsrefs}

\usepackage[makeroom]{cancel}
\usepackage{setspace}
\usepackage{amsfonts} 
\usepackage{amssymb}
\usepackage{amsthm}
\usepackage{indentfirst}
\usepackage{enumerate}
 \usepackage{enumitem}   
\usepackage{amsmath}
\usepackage{graphicx}

\usepackage{mathabx}
\graphicspath{ {./images/} }

\usepackage{enumitem}
\usepackage[most]{tcolorbox}

\usepackage{tikz-cd}

\definecolor{myGreen}{RGB}{10 100 10}

\newcommand{\Id}{\text{Id}}

\newcommand{\ip}[2]{\langle #1, #2 \rangle}

\newcommand{\bC}{\mathbb{C}}
\newcommand{\bI}{\mathbb{I}}
\newcommand{\bJ}{\mathbb{J}}
\newcommand{\bH}{\mathbb{H}}

\newcommand{\bB}{\mathbb{B}}
\newcommand{\bR}{\mathbb{R}}

\newcommand{\bF}{\mathbb{F}}
\newcommand{\bN}{\mathbb{N}}

\newcommand{\cA}{\mathcal{A}}

\newcommand{\cM}{\mathcal{M}}

\newcommand{\cS}{\mathcal{S}}

\newcommand*{\tran}{\mathsf{T}}
\newcommand*{\re}{\text{Re }}
\newcommand*{\Rea}{{\mathfrak R}}
\newcommand*{\im}{\text{Im }}
\newcommand*{\ncS}{\mathrm{ncS}}
\newcommand*{\ncconv}{\mathrm{ncconv}}

\newtheorem{theorem}{Theorem}[section]
\newtheorem{lemma}[theorem]{Lemma}
\newtheorem{prop}[theorem]{Proposition}

\newtheorem{cor}[theorem]{Corollary}
\newtheorem{thm}[theorem]{Theorem}

\newtheorem*{cor*}{Corollary}
\newtheorem*{thm*}{Theorem}
\newtheorem*{lem*}{Lemma}

\newtheorem*{prop*}{Proposition}

\theoremstyle{definition}

\newtheorem{example}[theorem]{Example}

\newtheorem*{defn*}{Definition}

\theoremstyle{remark}

\title{\textbf{Real Noncommutative Convexity I}}

\author{David P. Blecher}
	
\address{Department of Mathematics, University of Houston, Houston, TX 77204-3008.}
	
\email{dpbleche@central.uh.edu}

\author{Caleb Becker McClure}

\address{Department of Mathematics, University of Houston, Houston, TX 77204-3008.}
	
\email{cbmcclur@central.uh.edu}

 \subjclass[2020]{47A20, 46A55, 47L07; Secondary 46L51, 46L52, 47L05, 47L25, 47L30}

\begin{document}
    \begin{abstract} 
We   initiate the theory of real noncommutative (nc) convex sets, the real case of the recent and profound complex theory   developed by  Davidson and Kennedy.   The present paper focuses on the real case of the topics from the first several sections of their memoir \cite{DK}.   Later results  will be discussed in future papers.    We develop here some of the infrastructure of  real nc convexity, giving many foundational structural results for real operator systems and their associated nc convex sets, and 
elucidate how the complexification interacts with the basic convexity theory constructions.   
Several  new features appear in the real case, including the novel notion of the complexification of a nc convex set. 
  \end{abstract} \maketitle

\begin{center} {\em In memory of Zhong-Jin Ruan (1955-2025)} \end{center} 

            \section{Introduction}

We open 
 with the categorical duality of compact convex sets and operator systems, beginning with the classical case. 
              Let $X$ be a compact Hausdorff space. A {\em concrete function system} is a selfadjoint unital subspace  $V$ of $C(X)$, where $C(X)$ is the abelian $C^*$-algebra of scalar valued continuous functions on $X$.  Our scalar field $\bF$ will  either be  $\bR$ or $\bC$.  An element $f \in V$ is called positive if for all $x \in X$ we have $f(x) \geq 0$. A state on $V$ is a scalar valued linear functional on $V$ which is unital, selfadjoint, and positive in the sense that it maps positive functions to positive numbers.
        Equivalently, these are the unital contractive functionals on $V$.  The collection of states on $V$, denoted by $S(V)$, is a convex set.
        By the Banach-Alaoglu theorem 
         it is also compact with the 
         weak* topology. Conversely, given a compact convex set $K$ the set $A(K)$ of affine scalar valued functions on  $K$ is a function system inside $C(K)$.  There is also an abstract characterization of function systems which will be discussed in a later section.
Kadison's representation theorem shows that there is a duality between function systems and compact convex sets. Indeed, we have that for all function systems $V$ and compact convex sets $K$
        \[A(S(V)) \cong V, \quad S(A(K)) \cong K , \]
        where the isomorphism in both cases is given by evaluation. The         maps taking $V \mapsto S(V)$ and $K \mapsto A(K)$ are contravariant functors,  so that the category of convex sets is dually equivalent to the category of function systems. All the above works  for  convex sets in both real or complex vector spaces, and for real or complex function systems (see e.g.\ \cites{Alfsen,PT}).         However               the relationship between the real and complex theory is complicated in places.  For example,  
        many convexity theoretic results about complex function spaces $V$ are  proved via Re$\, (V)$, i.e.\  in the real setting, since the direct complex variant can be messy.

        In 
        \cite{DK} Davidson and Kennedy establish a profound noncommutative (nc) convexity theory, in the complex case.   Early in this work they exhibit a categorical equivalence similar to the one above, but for  noncommutative convex sets and complex operator systems,  giving a non-commutative analogue of Kadison's duality. This is built on previous work of Webster and Winkler \cite{WW}, who use complex matrix convex sets instead.  
       Much of the   matrix  convexity  in this sense hitherto was done in the complex case.  
       See also \cite{Wittstock,CEinj} for  e.g.\ the original definition of matrix  convexity in the real case.  
       However there
        has been very substantial and remarkable 
        work from a different perspective  on this matrix convexity in the real case, some of it quite recent, motivated in part by connections to system engineering and the  matrix inequalities and convexity found there.  See e.g.\ \cite{HKM} and references therein, where the theory of matrix convex sets has been developed from the point of view of positivity domains of (affine linear) polynomial maps (note that all polynomially convex matrix convex sets are defined by an operator system).  For example the papers \cite{Ev,EH,JKMMP,KMS, CN}        develop many beautiful aspects of real matrix convexity, particularly for finite dimensional classes of particular interest like  spectrahedra and their variant of matrix convexity. 
       (We thank Scott McCullough and James Pascoe for discussions on this history and recent developments.)         Aspects of this work will definitely interact with our program in the future, and it furnishes very many interesting and important examples.    Indeed for this reason we will not spend much time on new examples in our paper: there is already a great abundance of interesting examples in that literature.  Many more may be derived from the huge existing classical convexity literature, as we hope to present elsewhere.     However even basic examples are very interesting as we shall see.

 Earlier this year the first author and Russell  developed the theory of real operator systems in \cite{BR}.  The present paper is a sequel to this, which in turn was a sequel of \cite{BReal}.  Given any real operator system, there is a very natural way to  complexify to get a complex operator system. The process of complexification is functorial in the sense that many of the constructions done with operator systems (for instance, max/min operator systems, duals, and tensor products) usually commute with complexification. This allows much of the complex theory of operator systems  and nc convex sets
 to be applied in the real case. 
        Therefore it is natural to ask if there is a categorical equivalence between real compact noncommutative (nc) convex sets and real operator systems?  More generally does the theory of Davidson and Kennedy         carry over to the real case?
         There are several motivations for investigating this.  For example, classical convexity is in many ways essentially a real theory, as inspection of foundational texts such as \cite{Alfsen, Phelps} shows, or e.g.\ as one sees in nearly all graduate courses on convexity theory, or in some aspects of the list in the previous paragraph.  Thus the real theory is likely to play a future role in operator algebras and mathematical physics.  As mentioned in more detail in the second paragraph of \cite{BR}, real structure occurs naturally and crucially
 in very many areas of mathematics and mathematical physics, and in several deep mathematical theories at some point a crucial advance has been made by switching to the real case (e.g.\ in $K$-theory and the Baum-Connes conjecture, see also for example\ \cite{Ros}).   This is sometimes because  the real category is bigger and hence allows more freedom.  In our case, every complex operator system (resp.\ nc convex set) is clearly a real operator system
 (resp.\ real nc convex set), but there are many interesting real operator systems (resp.\ nc convex sets) that are not complex operator systems (resp.\ complex nc convex sets).  An obvious example of this is the selfadjoint matrices:\ the real nc convex set associated with these is not a nc convex set in the sense of \cite{DK}.  This is a somewhat trivial example (much better examples may be retrieved from the list  in the last paragraph), but it illustrates the point. 
 
 In this paper we initiate the theory of real nc convex sets in the sense above, investigating the real case  of Davidson and Kennedy's theory, at least  up to  Chapter 5 of \cite{DK} 
(since our paper is already lengthy, later results will be discussed in future papers).  We also have many complementary results, and several  new features appear in the real case.  Many of these are connected to the fact that, as opposed to the classical case of convex sets and function systems, it turns out that there is a very natural way to complexify a nc convex set. Because this complexification is functorial this will give us an efficient way to generalize  the theory to the real case.  We give many foundational structural results for real operator systems and their associated nc convex sets.  In particular we elucidate how the complexification interacts with the basic convexity theory constructions. In addition, we include some results  about nc convex sets in the real and complex case that do not seem to be in the literature.   Most of our results and proofs about nc convex sets apply verbatim to matrix convex sets in an obvious way, but we  focus on the nc convex category.

 The differences between the real and complex case discovered in \cite{BR} show up for us too, here and in the sequel papers \cite{BMcII}.  
 For example, the casual reader may miss some of the following items, some of which also mention some of our novel contributions: 1)\ Some of the major techniques in the complex case are not  available in the real case.   For example,  a common trick 
 in \cite{DK} and its sequels, is that the complex nc affine function space $A(K)$ is  unitally order isomorphic to the classical affine function space on $K_1$ (the `first level' of $K$). In particular the classical state space faithfully reproduces the affine structure and order at level 1. 
  This is never true in the real case unless $K$ is what we call `symmetric', that is the corresponding real operator system has trivial (i.e.\ identity) involution  (see Theorem \ref{DK2.5.8} and 
Remark 2 after it). 
 This has significant ramifications for us;  we will often have to find other routes, here and in the sequels. 
This is related to 2),  the absence of a Min and Max functor for nonsymmetric  nc convex sets, e.g.\ those corresponding to general
        real operator systems such as the quaternions.   Indeed we shall see examples of minimal and maximal complex nc convex sets with a real counterpart  
    which  is not minimal nor maximal. 
        3)\ The use of complexification is usually pervasive for us as alluded to above.  We spend much time in our series of papers on developing  various technical features of the complexification, and developing   
        other effective tools in the real case.   4)\  There are places where the existing complex proofs do not work, and we will usually point such out.
        Sometimes one is then saved by complexification (point 3), as is the case for example in the proof of one of our categorical dualities.           Also `proof by complexification' is often much quicker, and sometimes safer.  5)\ 
   Even a cursory survey of our proofs will reveal that in the real case many calculations have their own technicalities which are quite different to the complex case.  
   As just one example of this, the reader will  notice in our papers a repeated use of a certain isometry $u$    and map $c$ 
   in many proofs.         
  
   With all this in mind we found it remarkable that things work out so well; so that we were able to establish the real case of so much of the complex theory.   
       The reader might keep in mind as illustrations when 
encountering the main results of our paper, a few basic examples (see Examples \ref{ConvexInterval1}, \ref{MatStates}, \ref{ConvexInterval2}, \ref{Quaternions}, \ref{ointi}, \ref{newbee}).   We often weave in such examples, however they may be exposited further  in a different section of the paper,  and 
 in later papers.    Our advances in the present paper are of course a primary foundation for the sequel papers.  In \cite{BMcII} we focus on the theory of nc extreme (and pure and maximal) points and the nc Choquet boundary in the real case, and on the theory of real nc convex and semicontinuous functions and real nc convex envelopes. Again our main emphasis there is on how these interact with complexification, for example complexifying nc convex functions and their convex envelopes.   A third paper under construction is focused on nc real Choquet theory and complementary topics.

        Turning to the structure of our paper, Section \ref{SecPreliminaries} gives some background on real and complex operator systems and their complexifications.
        In keeping with the task and nature of our paper we do however expect the reader to be reading alongside  with parts of \cite{BR,DK}.  
        The very recent textbook  \cite{Dav} contains an account of many results from \cite{DK}. 
        Because of this we also do not need to be very pedantic 
        or overly careful with definitions, preliminaries, or the history of the subject, which may usually be found in  detail in those sources.  In Section \ref{SecBasicDefinitions} we define real and complex non-commutative (nc) convex sets and give basic examples such as the real non-commutative state space. Section \ref{SecCompl} describes the complexification of a real nc convex set. This can be done intrinsically by specifying what elements will be in the complexification, or extrinsically by taking a suitable complex nc convex hull of a real nc convex set. These two constructions will be equivalent.   We show that there is a unique reasonable complexification of a real nc convex set.  We also prove functorial properties of the complexification. For instance, if $K$ is a nc convex set and $A(K)$ are the nc affine functions on $K$, then
        \[A(K_c) = A(K)_c .\]
        In Section \ref{SecCatDuality}, we show the real version of Davidson and Kennedy's categorical duality.  Some useful techniques here are using a real version of the nc separation theorem (see Theorem \ref{separationthrm}), or the functoriality of complexification from Section \ref{SecCompl}.  This has many applications.

        Section \ref{SecMinMax} begins with some facts about  function systems.
        A classical compact convex set $K$ may be turned into a compact nc complex convex set using the fact that the function system $A(K)$ can be given a minimum and maximum operator system structure OMIN and OMAX \cite{PTT,BXhabli,BR}, and then employing the categorical duality between compact nc convex sets and operator systems.  Thus we  define ${\rm Min}(K)$ and ${\rm Max}(K)$  via the relation
         $A({\rm Min}(K)) = {\rm OMIN}(A(K))$ and $A({\rm Max}(K)) = {\rm OMAX}(A(K))$.
        As with operator systems, the min and max structure given to a real convex set commutes with complexification. This process uses the bipolar of a nc convex set, and so we develop that in Section \ref{SecBipolarThrm}. Section \ref{SecNoncFunct} develops the important notion of {\em non-commutative functions} in the real case. 
                We may avoid many of the complications  in the proofs 
        by e.g.\  proving key theorems such as $4.3.3$ in \cite{DK} (or \cite[Theorem 16.8.10]{Dav}) in the real case by complexification.

        \section{Preliminaries}\label{SecPreliminaries}
        \subsection{Operator Systems and Operator Spaces}

        For general background on operator systems and spaces, and in particular on the definitions etc.\ in the rest of this section, we refer the reader to e.g.\ \cite{Pnbook,BLM,DK,Dav} and in the real case to e.g.\ \cite{BReal,BR}.  
It might also  be helpful to also browse some of the other  existing real operator space theory  e.g.\ \cite{ROnr,RComp,Sharma,BT,BCK}.  Some basic 
real $C^*$- and von Neumann algebra theory may be found in \cite{Li}.  

We write $M_n(\bR)$ for the real $n \times n$ matrices, or sometimes simply $M_n$ when the context is clear. 
  Similarly in the complex case. 
We sometimes use the quaternions $\bH$ as an example: this is simultaneously  a real operator system, a real Hilbert space, and a real C$^*$-algebra, usually thought of as a real $*$-subalgebra of $M_4(\bR)$ or $M_2(\bC)$. Its complexification is $M_2(\bC)$.  
 The letters $H, L$ are usually reserved for real or complex Hilbert spaces, and $K$ for (nc) compact convex sets.  Every complex Hilbert space $H$ is a real Hilbert space, i.e.\ we forget the complex structure.   More generally we write $X_r$ 
 for a complex Banach space regarded as a real Banach space. 
 We write $X_{\rm sa}$ for the selfadjoint elements  
 in a $*$-vector space $X$.  In the complex case 
 $M_n(X)_{\rm sa} \cong (M_n)_{\rm sa} \otimes X_{\rm sa}$, but this  fails for real spaces. A subspace of $B(H)$ is {\em unital} if contains the identity,   
and a map $T$ is unital if $T(1) = 1$.  Our identities $1$ always have norm $1$.  We write $\Rea \, a$ for $\frac{1}{2}(a + a^*)$, while for 
$z \in M_n(\bC)$ we write $\re z$ for $x \in M_n(\bR)$ where $z = x + iy$ for  $y \in M_n(\bR)$. Finally, for a cardinal $n$ we define the isometry $u_n = \frac{1}{\sqrt{2}} \begin{bmatrix} 
    1_n \\ -i \cdot 1_n 
\end{bmatrix}$ where $1_n$ is the $n$-dimensional identity operator.  We sometimes also write this as $u$.

 A {\em concrete complex (resp.\ real) operator system} $V$ is a unital selfadjoint subspace of $B(H)$ for $H$ a complex (resp.\ real) Hilbert space. For $n \in \bN$ we have the identification $M_n(B(H)) \cong B(H^{(n)})$ where $H^{(n)}$ is the $n$-fold direct sum of $H$. From this identification, $M_n(V)$ inherits a norm and positive cone. The latter is the set $M_n(V)^+ \coloneq \{x \in M_n(V): x = x^* \geq 0 $ in $ B(H^{(n)})\}$. For  $n \in \bN$ we define the amplification of a linear map $\varphi: V \to W$ by 
        \[\varphi^{(n)}: M_n(V) \to M_n(W)\]
        \[[x_{ij}] \mapsto [\varphi(x_{ij})] . \]
        The natural morphisms between operator systems are {\em  unital completely positive} (ucp) functions, which are linear maps $\varphi: V \to W$ that are unital and every amplification is positive (or equivalently selfadjoint and contractive). The isomorphisms (resp.\ embeddings) of operator systems which are used in this paper are bijective (resp.\ injective) ucp maps whose inverse (resp.\ inverse in its range) is ucp.  These are called unital complete order isomorphisms (resp.\ unital complete order embeddings); or {\rm ucoi} (resp.\ ucoe) for short.

        Similarly a concrete operator space $E$ is a subspace of $B(H)$ with norms on $M_n(E)$ inherited from $B(H^{(n)})$. The natural morphisms between operator spaces are the completely bounded maps, namely the linear maps $\varphi$ between operator spaces such that the amplifications of $\varphi$ are uniformly bounded. If the uniform bound is $\leq 1$ then $\varphi$ is called a {\em complete contraction}. If the amplifications of $\varphi$ are isometries then $\varphi$ is a {\em complete isometry}.  
     For a possibly infinite cardinal $n$, $M_n(E)$ is the space of matrices whose ‘finitely supported' submatrices have uniformly bounded norm. In the case $E$ is the scalar field we simply write $M_n$;  thus 
$M_n \cong B(l^2_n)$. Indeed for every Hilbert space $H$, $B(H) \cong M_n$ $*$-isomorphically for some $n$, after one chooses an orthonormal basis.
   The above definitions hold for both real and complex operator systems. 
        
        There are abstract characterizations of real/complex operator spaces and operator systems. An abstract real/complex operator space is a vector space $E$ with a sequence of matrix norms $\{||\cdot||_n\}_{n=1}^\infty$ satisfying Ruan's axioms.  Operator systems on the other hand         are usually abstractly characterized using the 
        notion of an {\em order unit} $e$: these satisfy  that for any selfadjoint elements $x$ there is a $t > 0$ such that $x + te \geq 0$.  We say that $e$ is {\em archimedean} if  $x+ \epsilon e \geq 0$ for  all $\epsilon > 0$ implies $x \geq 0$.
        For abstract complex operator systems, we begin with a complex $*$-vector space $V$ (a vector space with a period 2 conjugate linear map $*: V \to V$), with a {\em matrix ordering} $M_n(V)^{+}$ and an {\em archimedean matrix order unit} (or {\em AOU}) $e$. The definition is the same in the real case, with conjugate linear replaced by linear.  The matrix ordering consists of  cones $M_n(V)^+$ in the $n \times n$ matrices of $V$,  which are selfadjoint, proper, and closed under compressions by matrices $\beta \in M_{n,m}(\bC)$.  An archimedean matrix order unit is an element $e \in V$ such that $e \otimes 1_n$ (where $1_n$ is the identity of the $n\times n$ matrices) is an archimedean order unit for each $n$. These conditions define an abstract operator system.  One may then prove that there is a unital complete order embedding of this space into $B(H)$ for some $H$.  
        See e.g.\ \cite{Dav, Pnbook}, or \cite{BR} in the real case. 

        A real operator system can naturally be made into a complex operator system by complexification. To do this, we start with a real abstract operator system, call it $V$, with involution $*$, matrix ordering $M_{n}(V)^{+}$, and Archimedean matrix order unit $e$. The complexification of $V$ is the complex vector space $V_c$ consisting of elements $x+iy$ for $x,y \in V$. We give this a conjugate linear involution $(x+iy)^* = x^* - i y^*$. The matrix ordering $M_n(V_c)^+$ will be defined by
        \[M_n(V_c)^+ = \{x+iy \in M_n(V_c): c(x,y) \geq 0\}\]
        where $$c(x,y) = \begin{bmatrix}
            x & -y\\ y & x
        \end{bmatrix} .$$ 
  We also sometimes write $c(x+iy)$ for $c(x,y)$.       The element $e+i0$ will be an archimedean order unit. With this, the complexification becomes an abstract complex operator system.  Moreover one can show that this operator system structure on the complexification is the unique one satisfying Ruan's completely {\em reasonable} condition, namely that the map $\theta_V(x+iy)=x-iy$ is a ucoi (or equivalently, is completely contractive).

        In part of \cite{BR} it was checked that many of the basic theorems and constructions for complex operator systems also hold for real operator systems. Very many foundational
structural results for real operator systems were developed, and it was shown  how the complexification interacts
with the basic constructions in the subject. 
       In certain parts of our paper we will need real operator systems $V$ with trivial involution $x^*=x$.
         It is easy to see that these coincide with the operator systems which are the selfadjoint part of complex operator systems. (Or, of real operator systems.) Thus they form an important class of real operator systems. Note however that in this case $M_n(V)$ has a nontrivial involution for $n \geq 2$, the transpose.
         At the end of Section 4 we characterize the nc convex sets associated with such operator systems: they are the {\em symmetric} nc convex sets.
         Just as `level 1' (that is, $M_1(V)$) of the complex operator systems are exactly the complex function systems 
         (see e.g.\ Section 4.3 of 
\cite{KRI}) and \cite{PT,PTT}, 
the real function systems (or real (unital) function spaces) are exactly `level 1' of the real operator systems with trivial (i.e.\ identity) involution \cite[Section 9]{BR}.  
 It is shown in \cite{BR} that a real operator system $V$ can be given a
minimum or a maximum operator system structure if and only if $V$ has trivial
 involution.  
  More precisely,
 by Remarks after Examples 9.6 and 9.14 in that paper, and by  \cite[Proposition 9.19]{BR}, OMAX$(V )$and OMIN$(V )$ are operator systems if and only if the involution on the real operator system $V$ is trivial, i.e.\ the identity. 
 
	\subsection{Noncommutative real convex sets and affine functions}     \label{SecBasicDefinitions}

As stated in e.g.\ \cite{BReal,BR}, every positive functional on a real operator system is a multiple of a state, and every contractive unital functional is a state.  The norm of a positive functional (resp.\ cb norm of a completely positive map) is its (resp.\ the norm of its) value at 1.  The real states $\varphi$ on a real operator system $V$ are precisely the real parts of complex states on $V_c$ (such as $\varphi_c$), or of a complex $C^*$-algebra generated by $V_c$. However, the real parts of two different such complex states may coincide on $V$.  Similarly, the real matrix states $\varphi$ on $V$ are precisely the `real parts' $\re \circ \psi$ of complex matrix states $\psi$ on $V_c$ (such as $\varphi_c$).

      See \cite[Section 1.3 and 1.6.1--1.6.4]{BLM} for  basics about dual operator spaces and their theory.  The real case is almost identical (see e.g\ \cite{BReal}).    We say a little more about the 
      weak* topology: For $E$ a real dual operator space space with operator space predual $E_*$, $M_n(E)$ is also a dual operator space, and we have $$M_n(E) \cong M_n(CB(E_*, \bR)) \cong CB(E_*, M_n)$$
        So, for $[f_{st}^{\alpha}] \in M_n(CB(E_*, \bR))$ and $[f_{st}] \in M_n(CB(E_*, \bR))$ we have that $[f^{\alpha}_{st}] \to [f_{st}]$ weak* if and only if for all $[x_{kl}] \in M_{m}(E_*)$ we have $[f_{st}^{\alpha}(x_{kl})] \to [f_{st}(x_{kl})]$ in $M_{nm}$.
        Also, $M_n(E_*) \cong w*CB(E,M_n)$, the weak* continuous completely bounded maps.

        For a real operator space $E$, as in \cite{DK} we  define $\cM(E) = \bigsqcup_n \, M_n(E)$ (for $n$ cardinals bounded by some cardinal $\kappa$) with $M_n(E)$ the matrix space of $E$.  Here $\bigsqcup_n$ is the disjoint union.
        For $X \subseteq \cM(E)$ define $X_n = X \bigcap M_n(E)$. We call this the nth level of the nc set $X$. 
In the case $E = \bR$ write $\cM = \cM(\bR)$. A {\em real non-commutative convex set} over $E$ is a subset $K = \bigsqcup K_n \subseteq \cM(E)$ such that

        \begin{enumerate}
            \item $K$ is graded: $K_n \subseteq M_n(E)$ for all $n$
            \item Closed under direct sums: $\sum \alpha_i x_i \alpha_i^\tran \in K_n$ for all bounded families $\{x_i \in K_{n_i}\}$ and every family of isometries $\{\alpha_i \in M_{n,n_i}\}$ where $\sum \alpha_i \alpha_i^\tran = 1_n$.
            \item Closed under compressions: $\beta^\tran x \beta \in K_m$ for every $x \in K_n$ and isometry $\beta \in M_{n,m}$.
        \end{enumerate} 
        As in \cite{DK} we say that $K$ is {\em closed/compact} if $E$ is a dual operator space and $K_n$ is closed/compact in the         weak* topology in $M_n(E)$. 

        For $\{x_i \in M_{n_i}(E)\}$ bounded and $\{\alpha_i \in M_{n_i,n}(\bR)\}$ such that $\sum \alpha_i^\tran \alpha_i = 1_n$, a {\em nc convex combination} of $x_i$ is defined as $\sum \alpha_i^\tran x_i \alpha_i \in M_n(E)$. As in the complex case (see Proposition 2.2.8 in \cite{DK}
        or Remark 16.4.3 (3) in \cite{Dav}) a subset $K \subseteq \cM(E)$ is nc convex if and only if it is closed under nc convex combinations.

        \begin{example}
        \label{ConvexInterval1}
            Let $a,b \in \bR$ with $a < b$. Then for $n \in \bN$ let $K_n = [a1_n, b1_n]$ where
            $$[a1_n, b1_n] = \{\alpha \in (M_n(\bR))_{sa}: a 1_n \leq \alpha \leq b 1_n\}.$$  Then 
            $K = \bigsqcup_{n \in \bN} K_n$ is a real compact convex set over $\bR$ called the {\em real compact operator interval},            and written as $\bI = [[a,b]]$. If we replace $\bR$ in the above definition with $\bC$ then we get the {\em complex compact operator interval}. We have $M_n(\bR) \subseteq M_n(\bC)$ and $M_n(\bR)^+ \subseteq M_n(\bC)^+$ and therefore the real operator interval is a subset of the complex operator interval. Here we see that $K_1$ agree in the real and complex case because $K_1$ is the interval $[a,b] \subseteq \bR$.
        \end{example}
        
                \begin{example}
        \label{MatStates}
            Let $V$ be a real operator system, then the {\em real nc state space} $K = \ncS(V)= \bigsqcup_n \mathrm{UCP}(V,M_n(\bR))$ is a point-weak$^*$ compact nc convex set over $V^*$. 
            This norms $V$.  Indeed as in the complex case 
            $$\| [v_{ij} ] \| = \sup
            \{ \| [\varphi(v_{ij}) ] \| : n \in \bN, \varphi \in \mathrm{UCP}(V,M_n(\bR)) 
            \} . 
            $$  To see this quickly note that taking a ucoe  $\varphi : V \to B(H) \cong M_\kappa$ does this in one shot.
            For finite dimensional subspaces $K$ of $H$
            the compressions $P_K \varphi(\cdot)_{|K}$  achieve in the limit the norm above, identifying  $B(K) \cong M_n$.                \end{example}

  \begin{example}
        \label{Matball}
    For a real dual operator space $E$,   the `sequence' of matrix unit  balls clearly constitute a nc convex matrix set $\bB(E)$.
      \end{example}
     
        The following proposition is useful for extending arguments about matrix convex sets to non-commutative convex sets.
        
        \begin{prop} \label{dum} Suppose we have $K$ a closed nc convex set over a dual operator space $E$, a net $\{x_i \in K_{n_i}\}$, and a net of isometries $\{\alpha_i \in \cM_{n, n_i}\}$ such that $\lim \alpha_i \alpha_i^\tran = 1_n$ and $\lim \alpha_i x_i \alpha_i^\tran = x \in M_n(E)$. Then, $x \in K_n$.
        \end{prop}
        \begin{proof}
            Same as in the complex case. See Proposition 2.2.9 in \cite{DK}. 
        \end{proof}
        This result implies as in  \cite[Proposition 16.4.5]{Dav} (or \cite[Proposition 2.2.10]{DK}) that
        
       \begin{prop} \label{nf} Suppose that $K$ and $L$ are closed nc convex set over a dual operator space $E$.  If $K_n = L_n$ for all $n < \infty$ then $K=L$.
        \end{prop}

        The natural morphisms between real nc convex sets are {\em real nc affine functions}. These will be maps $\theta: K \to L$ between real nc convex sets which are graded, respect direct sums, and equivariant with respect to isometries. That is, for all $n$
        \begin{enumerate}
            \item $\theta(K_n) \subseteq L_n$,
            \item $\theta(\sum \alpha_i x_i \alpha_i^\tran) = \sum \alpha_i \theta(x_i) \alpha_i^\tran$ for all bounded families $\{x_i \in K_{n_i}\}$ and every family of isometries $\{\alpha_i \in M_{n,n_i}\}$ where $\sum \alpha_i \alpha_i^\tran = 1_n$,
            \item $\theta(\beta^\tran x \beta) = \beta^\tran \theta(x) \beta$ for every $x \in K_n$ and isometry $\beta \in M_{n,m}$.
        \end{enumerate}
    We say that     $\theta$ is continuous if $\theta|_{K_n}: K_n \to M_n(\bR)$ is continuous for every $n$. $A(K)$ is the space of all continuous affine nc functions from $K$ into $\cM(\bR)$.
        
        For $K,L$ classical convex sets and for $\varphi: K \to L$ a bijective function which is affine,  its inverse is easily seen to be affine. The same will be true in the non-commutative case. Indeed, $\varphi^{-1}$ is graded because for $l \in M_n(L)$ then $\varphi^{-1}(l) \in M_n(K)$, and for $y \in L_n$ and an isometry $\beta \in M_{n,m}$ then $\beta^\tran y \beta = \beta^\tran \varphi(\varphi^{-1}(y)) \beta = \varphi(\beta \varphi^{-1}(y) \beta)$. Taking $\varphi^{-1}$ of the left and right hand sides gives the result. A similar proof holds to show $(2)$.

        \begin{subsection}{Some relations between the real and complex case}

        \begin{lemma}\label{reisAffine}
                The function $\re: \cM(\bC) \to \cM(\bR)$  taking a complex matrix $A+iB$ (where $A,B \in M_n(\bR))$ to the real matrix $A$ is real affine and completely contractive. The same is true for the map $\im: \cM(\bC) \to \cM(\bR)$ taking $a+ib$ to $b$.
        \end{lemma}
        \begin{proof} The map 
            $\re$ is well defined and graded because it sends $N\times N$ matrices over the complex numbers to $N\times N$ matrices over the reals. 
           We leave it to the reader  that it is completely contractive. To show (3) in the definition of `affine', let $A+iB \in M_N(\bC)$ and $\beta \in M_{N,M}(\bR)$ be an isometry. Then we have
            \begin{align*}
                \re(\beta^\tran (A+iB) \beta)
                &= \re(\beta^\tran A \beta + i \beta^\tran B \beta)
                \\&= \beta^\tran A \beta = \beta^\tran \re(A+iB) \beta
            \end{align*}  The same proof holds to show condition (2) and for the imaginary part.
        \end{proof}
        
If $V = M_n(\bC)_{\rm sa}$  then $V_c$ may be identified 
(via a unital complex complete order isomorphism, and of course complete isometry) with a canonical subspace of
$M_ n(\bC) \oplus M_ n(\bC)$.   Namely 
if $x,y \in V$ then $z = x + iy$ in $V_c$ is   identified with $(z, z^T)  \in M_ n(\bC) \oplus M_ n(\bC)$. 
For a general complex operator system $W,$ if $V = W_{\rm sa}$ then the canonical complex linear map
$u : V_c \to W$ is  an isometric  and unital
 identification, since  $x + iy$ for $x, y \in V$ in both cases may be identified with $c(x,y) \in M_2(W)$.
  For a complex selfadjoint  operator space $E$, we claim that $(E_{\rm sa})_c = E^{\rm sym}$ completely isometrically.  Here $E^{\rm sym}$ is $E$ but with matrix norm $\| [x_{ij} ] \|_{\rm sym} = \max \{  \| [x_{ij} ] \| , \| [x_{ji} ] \| \}$.  Indeed 
$E_{\rm sa} \subset E^{\rm sym}$ and $E_{\rm sa} + i E_{\rm sa} = E^{\rm sym}$.   The map 
$x \mapsto x^*$ on $E^{\rm sym}$ is a period 2, conjugate linear, complete isometry with fixed points $E_{\rm sa}$.
So $(E_{\rm sa})_c = E^{\rm sym}$.   Moreover this identification is as operator systems if $E$ is an operator system.

\begin{lemma}\label{ccl}  For a complex operator system  $V$
the `identity map' taking 
$x + iy \in (V_{\rm sa})_c$ to $x + i y  \in V$, for $x, y \in V_{\rm sa}$, is ucp, and is a complex linear bijective isometric order isomorphism.            \end{lemma}

            \begin{proof}  The complexification of the inclusion $V_{\rm sa} \to V$ is a 
canonical unital completely isometric complex map $(V_{\rm sa})_c \to V_c$.  If we compose this
with the canonical complex quotient map $V_c \to V$ (see e.g.\ the third paragraph of
\cite[Section 11]{BR}), we obtain a ucp map $(V_{\rm sa})_c \to V$.  This agrees with the `identity map'. It clearly is complex linear and bijective. To see that it is an  isometric order isomorphism  note that   $x + iy$ for $x, y \in V_{\rm sa}$ in both cases may be identified with $c(x,y) \in M_2(V)$.  \end{proof}

Note that the above is an order isomorphism, but not necessarily a complete order isomorphism.

\begin{lemma}\label{cspc}  For a complex operator system  $V$ the complex nc state space of $V$ is real nc affinely homeomorphic to the 
closed nc subset $\{ \varphi \in {\rm ncS}_{\bR}(V) : \varphi(i1) = 0 \}$ in the real nc state space, via the ‘real part' operation.  \end{lemma}

            \begin{proof} Since $\re$ is completely contractive and nc affine by  Lemma \ref{reisAffine}, $\varphi \mapsto \re  \circ \varphi$ is a continuous nc affine map ncS$_{\bC}(V) \to {\rm ncS}_{\bR}(V)$.  
                         Conversely, we define a 
             map $\epsilon : {\rm ncS}_{\bR}(V) \to {\rm ncS}_{\bC}(V)$ 
            by  $$\epsilon(\varphi)(x) = 
            \varphi(x) - i \varphi(ix) = 2 (\varphi_c \circ j)(x), \qquad \varphi \in {\rm ncS}_{\bR}(V) x \in V,$$ where 
            $j : V \to V_c$ is the canonical complex linear inclusion (discussed e.g.\ in the third paragraph of \cite[Section 11]{BR}).
            It is easy to check that $\epsilon(\varphi)$ is selfadjoint and completely positive (since $j$ and $\varphi_c$ are),  and that $\re  \circ  \epsilon(\varphi) = \varphi$.  So  if $\varphi(i1) = 0$ then $\epsilon(\varphi) \in {\rm ncS}_{\bC}(V)$.  
            Clearly $\epsilon$ is a continuous nc affine map ncS$_{\bC}(V) \to {\rm ncS}_{\bR}(V)$.  
            \end{proof}
      \end{subsection}

        \begin{subsection}{Affine maps as an operator system}\label{AffAreOpSys}
        Let $K$ be a real compact nc convex set and let $A(K)$ be the collection of continuous real nc affine maps from $K$ to $\cM(\bR)$. As in the complex case this is a $*$-vector space with adjoint given by $f^*(k) = f(k)^\tran$ for $f \in A(K)$ and $k \in K$. 
        We identify $M_n(A(K))$ and $A(K,M_n(\bR))$.  We define  the positive cone $M_n(A(K))^+$ by saying $[f_{ij}] \in M_n(A(K))$ is positive if and only if $[f_{ij}(k)]$ is positive for all $k \in K$. 
        
        The matrix order unit  will be the constant function $1$ which sends everything in $K$ to the corresponding identity in $\cM(\bR)$. This is a matrix order unit because an element $a$ of $M_n(A(K))_{\rm sa}$ will be bounded by a number $c < \infty$ by the proof of \cite[Proposition 2.5.3]{DK} 
        (Remark 16.6.2 (2) in \cite{Dav}), which is the same in the real case.  Hence $-c I \leq a(k) \leq c I$
        for each $k \in K$, so that $c1 - a \geq 0$.         Suppose that for some $n \in \bN$ and $f \in M_n(A(K))_{\rm sa}$ we have $f+\epsilon 1_n \geq 0$ for all $\epsilon$.  Evaluating at  $k \in K$ we get $f(k) + \epsilon I \geq 0$. Taking $\epsilon$ to zero shows that $f(k) \geq 0$ for all $k$, and so $f$ is positive. Therefore, $1$ is an archimedian matrix order unit and $A(K)$ is a real operator system.

                    \bigskip

    {\bf Remark.}
  Because $M_n(\bR) \subseteq M_n(\bC)$, every complex nc convex set $K \subseteq \bigsqcup M_n(E)$ can be regarded as a real nc convex set $K \subseteq \bigsqcup M_n(E_r)$, where  $E_r$ is $E$ regarded as a real vector space.  
            Note that in this case complex affine functions with domain $K$ are real affine.
            We saw that the real nc affine functions on $K$ are a real operator system, and it contains the  complex affine functions as a real subsystem.

               \end{subsection}

        \section{Complexification}\label{SecCompl}
            Given an operator space $E$, we define its complexification $E_c$ to have matrix norms $M_n(E_c)$ inherited from the embedding $c: M_n(E_c) \to M_{2n}(E)$
            \[c : [x_{nm}+iy_{nm}] \mapsto \begin{bmatrix}[x_{nm}] & -[y_{nm}] \\ [y_{nm}] & [x_{nm}] \end{bmatrix}\]
            If $E$ is a dual space then $E_c$ will be too because $E_c = ((E_*)^*)_c = ((E_*)_c)^*$, and then it is easy to see that $c$ is a bicontinuous embedding for the weak* topologies. 
            
            Let $K$ be a real nc convex subset of $E$. Define the {\em complexification of K} as the set $K_c \subseteq \bigsqcup M_n(E_c)$ by $[z_{ij}] \in (K_c)_n$ if and only if $c([z_{ij}]) \in 
            K_{2n}$.

            \begin{theorem}\label{ncrealaffinecoml}
                Given a real nc convex set $K \subseteq \bigsqcup E_n$, the complexification $K_c \subseteq \bigsqcup (E_c)_n$ is a complex nc convex set with $K$ canonically embedded in $K_c$ via a real continuous nc affine map $\iota$.  We have that $K_c = {\rm co}_{\bC}(\iota(K))$, the noncommutative convex hull. 
                Also, 
                $x + iy \in K_c$ if and only if $x-iy \in K_c$, for $x, y \in M_n(E)$. 
                Moreover if $E$ is a dual operator space
                then  $K$ is closed (resp.\ compact) if and only if $K_c$ is closed (resp.\ compact). 
            \end{theorem}

            \begin{proof}
                Clearly $K_c$ is graded. To show (2) and (3) we need the map $c$ to behave well. Specifically, if $[x_{nm} + y_{nm}] \in (K_c)_{N}$ and $[a_{nm} + i b_{nm}] \in M_{K,N}(\bC)$ where $a_{nm}, b_{nm} \in \bR$, then
            \begin{align*}
                c([a_{nm} + i b_{nm}][x_{nm} + i y_{nm}]) 
                &= c([\sum_k a_{nk} x_{km} - b_{nk} y_{km} + i b_{nk} x_{km} + i a_{nk} y_{km}])
                \\&= \sum_k \begin{bmatrix}
                    [a_{nk} x_{km} - b_{nk} y_{km}]  & -[b_{nk}x_{km} + a_{nk} y_{km}]\\
                    [b_{nk}x_{km} + a_{nk} y_{km}] & [a_{nk} x_{km} - b_{nk} y_{km}]
                \end{bmatrix}
                \\&= \begin{bmatrix}
                    [a_{nm}] & - [b_{nm}] \\
                    [b_{nm}] & [a_{nm}]
                \end{bmatrix}
                \begin{bmatrix}
                   [ x_{nm}] & -[y_{nm}]\\ [y_{nm}] & [x_{nm}]
                \end{bmatrix}
                \\&= c([a_{nm} + i b_{nm}]) c([x_{nm} + iy_{nm}]).
            \end{align*}
            We also have
            \begin{align*}
                c([a_{nm} + i b_{nm}]^*) 
                &= c([a_{mn} - i b_{mn}])
                \\&=  \begin{bmatrix}
                    [a_{mn}] & [b_{mn}] \\
                    -[b_{mn}] & [a_{mn}]
                \end{bmatrix}
                \\&= \begin{bmatrix}
                    [a_{nm}]^\tran & [b_{nm}]^\tran \\
                    -[b_{nm}]^\tran & [a_{nm}]^\tran
                \end{bmatrix}
                \\&= c([a_{nm} + i b_{nm}])^\tran
            \end{align*}
            Let $x_i \in (K_c)_{n_i}$ and $\alpha_i \in M_{n,n_i}(\bC)$ be a family of isometries such that $\sum \alpha_i \alpha_i^* = 1_n$. Then we have $c(\sum \alpha_i x_i \alpha_i^*) = \sum c(\alpha_i) c(x_i) c(\alpha_i)^\tran$ where $c(\alpha_i)$ will be a family of real isometries such that $c(\alpha_i) c(\alpha_i)^\tran$ sum to 1. So, $c(\sum \alpha_i x_i \alpha_i^*) \in K_{2n}$ which means $\sum \alpha_i x_i \alpha_i^* \in (K_c)_n$.  This verifies condition (2). A similar proof works for condition (3) showing $K_c$ is a complex noncommutative convex set. Therefore, the complexification of a real nc convex set is a complex nc convex set.

            The map $\iota: K \hookrightarrow K_c$ taking $[x_{nm}] \mapsto [x_{nm} + i \, 0]$ is a real continuous nc affine map. Indeed, it is graded and satisfies properties (2) and (3) in the affine map definition because $\iota(\beta^\tran x \beta) = \beta^\tran \iota(x) \beta$, where the latter are viewed as elements of $K_c$. This map is well-defined because if $x \in K$ then $c(x+i0) = x \oplus x \in K$ and therefore $x+i0 \in K_c$.  It is also continuous at each level. 

            Clearly 
            ${\rm co}_{\bC}(\iota(K)) \subseteq K_c$  since $K_c$ is convex and contains $\iota(K)$.  For the reverse inequality, if $x + iy \in K_c$ then $x  + iy = u^* i_{2n}(c(x,y)) u$ where $u$ is the 
                        usual isometry $\frac{1}{\sqrt{2}} [ I_n \;  -i I_n ]^\tran$.  This is a nc convex combination of an element from $\iota(K)$.

            The first `if and only if'
             is clear from the definitions, the nc convexity, and the fact that $c(x,-y) = w c(x,y) w$ where $w$ is the selfadjoint unitary $I \oplus (-I)$.

            Finally, suppose that  $E$ is a dual real operator space and each $K_n$ is closed in the weak* topology.  Suppose that $(x_t + i y_t)$ is a net in $(K_c)_n$ with
            weak* limit $x+iy$ in $M_n(E_c) \cong M_n(E)_c$.  
            Then $x_t \to x$ and $y_t \to y$ weak* (see \cite[Lemma 5.2]{BReal} and its proof).  So $(c(x_t,y_t))$ is a net in $K_{2n}$ with 
            weak* limit $c(x,y)$.  Thus  $c(x,y) \in K_{2n}$ and $x+iy \in (K_c)_n$. Thus $K_c$ is closed.
            A similar argument works for  compactness. 
            The converse is easier.  (E.g.\ suppose that $(x_t)$ is a net in $K_n$ with
            weak* limit $x$.  Then $(\iota(x_t))$ is a net
            in $(K_c)_n$, and $\iota(x_t) \to x$ weak* in $M_n(E_c)$.)
            \end{proof}

        {\bf Remarks.}  1)\ Similar considerations show that if $K_c$ is nc convex then so is $K$. Define $r : K_c \to K$ by $r(x+iy) = x$.   By simple calculations (similar to the last proof and the proof of Lemma \ref{reisAffine}) this is continuous  and real nc affine, and $r \circ \iota = I_K$.  
        
             \smallskip 
             
        2)\  If $E$ is a dual operator space
                then $c : K_c \to K$ 
                is a bicontinuous embedding satisfying (2) and (3) in the definition of an affine function.

             \medskip 
             
 We may thus define $\theta_K : K_c \to K_c$ as the restriction of the canonical period 2 automorphism $\theta_E$ of $E_c$ taking $x + iy \to x-iy$. Then $\theta_K$  is easily seen to be a period 2 real nc affine homeomorphism of $K_c$ whose fixed points are $K$.  Conversely if $C$ is a complex nc convex set in $E_c$ possessing a period 2 real nc affine homeomorphism, then the set $K$ of its fixed points is easily seen to be a real nc convex set.  
 
    \begin{lemma}\label{Affineext} For real  nc convex sets $K, L$, every real  nc affine map $f : K \to L$ has a unique complex  nc affine extension $f_c : K_c \to L_c$. 
            If $L$ is complex nc convex there is a unique complex  nc affine extension $K_c \to L$. These extensions are continuous if $f$ is continuous.

            In particular, every real  nc affine  isomorphism (resp.\ homeomorphism) $f : K \to L$ has a unique  complex  nc affine bijective (resp.\ homeomorphism) extension $f_c : K_c \to L_c$.
            \end{lemma}
            \begin{proof} Define $f_c(x+iy) = u^* f(c(x,y)) u$ if $x + iy \in (K_c)_n$, 
            where $u = u_n$ is the isometry above ($u_n = 1/\sqrt{2} [  I_n \; -i I_n ]^\tran$).   Note that $c(\beta) u_n = u_m \beta$ for $\beta \in M_{m,n}(\bC)$. 
            Then $$f_c(\beta^* (x+iy) \beta) = 
            u^* f(c(\beta^* (x+iy) \beta) u = u^* f(c(\beta^*) c(x,y) c(\beta)) u.$$ 
            Since $c(\beta^*) = c(\beta)^\tran$, and $c(\beta) u = u \beta$, we obtain $f_c(\beta^* (x+iy) \beta) = \beta^* f_c(x+iy) \beta$.  So $f_c$ is affine. 
            A similar argument works if $L$ is complex  nc affine.
            If $f$ is continuous and $x_n + i y_n \to x + iy$ in $(K_c)_n$ then $x_n \to x, y_n \to y, c(x_n,y_n) \to c(x,y)$.  So it is clear from the formula at the start of the proof that $f_c(x_n+iy_n) \to f_c(x+iy)$, hence $f_c$ is continuous.

            The uniqueness is clear from the above and the relation $f_c(x+iy) = f_c(u^* c(x,y) u)$.  The isomorphism case evidently follows.
            \end{proof}

         \begin{lemma}\label{Affineemb}   If $f : K \to L$ is  a one-to-one  continuous nc affine map between closed nc convex sets, and if $K$ is compact, then  $f_n$ is a homeomorphism onto its (compact) range for all $n$, and $f(K)$ is a compact nc convex set.
    \end{lemma}
            
            \begin{proof}  A continuous one-to-one map on a compact space is a homeomorphism onto its compact range. \end{proof}

        We call $f$ in the last result a {\em nc topological affine embedding}.  Note that if $f$ is a real nc topological affine embedding, then  $f_c$ is one-to-one and is a complex nc topological affine embedding.

We say that a  complex   compact nc convex set $L$ is an (abstract) {\em reasonable complexification} of a 
 real compact nc convex set $K$ if it 
(or more properly, $(L, \epsilon_L)$) satisfies any one of the equivalent conditions in the next result.
         
 \begin{thm} \label{unco}  Let $\epsilon = \epsilon_L : K \to L$ be a real nc  affine topological  embedding
 from a  real compact nc convex set  to a complex compact nc convex set, with $$L_n = \{ u^* \epsilon_L(c(x,y)) u : c(x,y) \in K_{2n} \}, $$ for each $n$.   (That is,  $L_n$ consists of the $x + iy$ for $c(x,y) \in K_{2n}$.)  
 The following statements are equivalent:
    \begin{enumerate}
        \item [{\rm (1)}] $L$ is a complex compact nc convex set in a complex space $F$, and  $F$ has a real subspace $Y$      with $Y \cap iY = 0$ such that  $\epsilon_L(K) \subset Y$.
        \item [{\rm (2)}]   The map  $$u^* \epsilon_L(c(x,y)) u \mapsto  x$$
      is well defined on $L$.        I.e.\ the `real part function' on $L$ is well defined. 
         \item [{\rm (3)}] The map $\theta_L : L \to L$ taking
       $$u^* \epsilon_L(c(x,y)) u \mapsto  u^* \epsilon_L(c(x,-y)) u, \qquad c(x,y) \in K_{2n}, $$
       is well defined.
          \item [{\rm (4)}]  The map $\theta_L$ is a  well defined period 2 real nc affine
        homeomorphism with  fixed point set   $\epsilon_L(K)$.
    \end{enumerate}  
    Up to real nc affine    homeomorphism there is a unique $L$ satisfying these conditions.
    That is, $K$ has a unique reasonable complexification.
    \end{thm}

            \begin{proof}  For  $\epsilon_L : K \to L$ as in the statement, let $\widetilde{\epsilon_L} : K_c \to L$ be 
             the continuous  nc affine extension from Lemma \ref{Affineext}, which we also write 
                 as $\tilde{\epsilon}$.  It satisfies  $$\tilde{\epsilon}(x + iy) = u^* \epsilon_L(c(x,y)) u , \qquad c(x,y) \in K_{2n} ,$$
             and in particular $\widetilde{\epsilon_L} \circ \iota =  \epsilon_L$. 
             Also   $\tilde{\epsilon}$ is surjective, since 
             $$u^* \epsilon_{L}(c(x,y)) u =  \tilde{\epsilon}(u^* i_{2n}(c(x,y)) u).$$
             We will show that $\widetilde{\epsilon_L}$ is one-to-one if and only if any one of conditions (1)--(4) hold.
             Indeed if $\widetilde{\epsilon_L}$ is one-to-one then it
             is a   nc affine homeomorphism by Lemma \ref{Affineemb}.   Thus $L \cong K_c$, via a map taking $\epsilon_L$ to the canonical embedding $K \to K_c$.  Hence (1)--(4) all hold since they hold for $K_c$ (e.g.\ in (1) one may take $Y = E$ and $F = E_c$, using notation from the definition of $K_c$ above). 
             
             Clearly (4) implies (3).    If (3) holds and $u^* \epsilon_L(c(x,y)) u = u^* \epsilon_L(c(x',y')) u$
             then applying $\theta_L$ to this condition gives  $u^* \epsilon_L(c(x,-y)) u = u^* \epsilon_L(c(x',-y')) u$.  Averaging these we obtain $$\epsilon_L(x) = u^* \epsilon_L(c(x,0)) u = u^* \epsilon_L(c(x',0)) u = \epsilon_L(x').$$ So $x = x'$.   Thus (2) holds.
              
                Since $\epsilon_L$ is nc affine we have $\epsilon_L(c(x,y)) = W^*  \epsilon_L(c(x,y)) W$ 
     where $W$ is the matrix with rows $[0,-I]$ and $[I , 0]$.   It follows that 
     $\epsilon_L(c(x,y)) =  c(a,b)$ for some $a,b$.  We have 
           $$a = \vec e_1^\tran \epsilon_L(c(x,y)) \vec e_1= \epsilon_L(\vec e_1^\tran c(x,y) \vec e_1) = \epsilon_L(x).$$
                Thus  $u^* \epsilon_L(c(x,y)) u = \epsilon_L(x) + i \, b$.  Supposing $u^* \epsilon_L(c(x,y)) u = u^* \epsilon_L(c(x',y')) u$,
                     $$u^* \epsilon_L(c(x', y')) u = \epsilon_L(x') + i \, z, \qquad {\rm if} \; \epsilon_L(c(x',y')) = c(\epsilon_L(x'),z).$$  If (2) holds then $x = x'$ and so $b = z$, so that $\epsilon_L(c(x,y)) = c(a,b) = \epsilon_L(c(x,y')),$ hence  $y = y'$.
     Thus $\widetilde{\epsilon_L}$ is one-to-one. 
Similarly, assuming (1) note that $\epsilon_L(c(x,y)) \in M_{2n}(Y)$, so that $b$ in the last lines is in $M_n(Y),$ as is $\epsilon(x)$.    Thus $\epsilon(x) + i b = \epsilon(x') + i z$ implies $b = z$ and $x = x'$, and $y = y'$.  
      \end{proof}
               
            Example    \ref{Quaternions} shows that complexification can be complicated, and for instance can change the first level of a nc convex set quite a bit. For now we give a simpler example.

            \begin{example}
            \label{ConvexInterval2}

            Consider the real compact operator interval from Example \ref{ConvexInterval1}. The complexification of the real operator interval will be the complex compact operator interval. To see this, first let $a,b$ be real numbers. Let $\bigsqcup [a1_n, b1_n]_{\bR}$ be the real operator interval and $\bigsqcup [a1_n,b1_n]_{\bC}$ be the complex operator interval. Take $x+iy \in [a1_n,b1_n]_c$ and we want to show it is in $[a 1_n, b 1_n]_{\bC}$. By the definition of the complexification, $c(x+iy) \in [a 1_{2n}, b1 _{2n}]$ and so we have
            \[a1_{2n} \leq \begin{bmatrix}
                x & -y \\ y & x
            \end{bmatrix} \leq b 1_{2n}\]
            However, the above will hold if and only if
            \[a 1_{n} \leq x+iy \leq b 1_{n}\]
            and therefore $x+iy \in [a 1_{2n}, b 1_{2n}]_\bC$. Conversely, if we take $z \in [a1_{n}, b 1_{n}]_\bC$ then $z$ can be written as the sum of a real and imaginary part, say $z=x+iy$, satisfying the centered equation above and therefore $z \in [a 1_{n}, b 1_n]_c$.
                          \end{example}

            As in Theorem 2.4.1 of \cite{DK} (Corollary 16.5.3 of \cite{Dav}) we have a real noncommutative separation theorem.

            \begin{theorem}\label{separationthrm}
                Let $K$ be a real closed nc convex set over a real dual operator space $E$. Suppose there is an $n$ and $y \in M_n(E)$ such that $y \not \in K_n$. Then there exists $\gamma \in M_n(\bR)_{\rm sa}$ and a normal completely bounded map $\varphi: E \to M_n(\bR)$ such that $\Rea \, \varphi_n(y) \not \leq  1_n \otimes \gamma$ but for all $p$ and $x \in K_p$ we have $\Rea \, \varphi_p(x) \leq 1_p \otimes \gamma$.  
                If $0_E \in K$ we can take $\gamma = 1_n$.
                If $E$ is a real operator system and $K \bigcup \{y\}$ consists of selfadjoint elements, then $\varphi$ can be chosen to be selfadjoint.
            \end{theorem}

            \begin{proof}   This follows  as in the complex case in \cite{DK} from the Effros-Winkler 
              separation theorem in \cite[Theorem 5.4]{EW} (see \cite[Theorem 16.5.2]{Dav}) for a different account of the latter).
              The real version of the latter is proved exactly as in the complex case (needing the real version of the GNS representation theorem \cite[Theorem 3.3.4]{Li} 
              applied to a faithful real state of $M_n$). 
            \end{proof}

            Let $V$ be a real operator system and consider $\ncS(V)_c$. This will be a complex nc convex set, as is $\ncS(V_c)$. There is a canonical map
            \[\psi: \ncS(V)_c \to \ncS(V_c)\]
            induced by the canonical isomorphism 
            $CB(V,W)_c \cong CB(V_c,W_c)$ (see for example \cite[Theorem 2.3]{BCK}).  Indeed for $x,y \in V$ and $f,g \in M_{n}(\mathrm{CB}(V,\bC))$ we have  \[\psi(f  + i g )(x+iy) = f (x) - g (y) + i f(y) + i g(x).\]
            The inverse map takes $u \in CB(V_c,W_c)$ to 
            ${\rm Re} \, u_{|V} + i {\rm Im} \, u_{|V}$, 
            where ${\rm Re}, {\rm Im}$ here denote the two canonical projections $W_c \to W$. This map will be called $\gamma$.

            \begin{lemma}\label{ncSCommC} The map 
                $\psi: \ncS(V)_c \to \ncS(V_c)$ is a bijective continuous complex affine nc function with continuous inverse $\gamma$.  
            \end{lemma}
            
             \begin{proof} For the reader's convenience and because we will need some of the details later, such as certain specific maps, we give two proofs, and mention a third.   Since these are closed nc convex sets in a dual operator space we may use 
                         the idea in \cite[Proposition 2.2.10]{DK} (or  \cite[Proposition 16.4.5]{Dav})
            to see that it suffices to check this at the $n$th level, for all $n \in \bN$. Since $\psi$ is a restriction of 
               the canonical isomorphism 
            $CB(V,\bR)_c \cong CB(V_c,\bC)$, 
            it is a continuous complex nc affine  function with continuous inverse.   To see that this takes
            $\ncS(V)_c$ onto $\ncS(V_c)$, note that by \cite[Lemma 3.1]{BR}  a map 
            $u : V_c \to M_n(\bC)$ is a complex matrix state
            if and only if its restriction to $V$ is  
            real ucp.  However the real ucp maps 
            $h: V \to M_n(\bC)$ are identifiable with the elements $f + ig \in \ncS(V)_c$.  
            To see this, notice that 
 the latter are precisely the  $f + ig$ such that  $c(f,g)(x) = c(f(x),g(x))$
            defines a real matrix state on $V$.
            Indeed these matrix states are precisely the ones which can be identified (via composition with the canonical identification  $c_n : M_n(\bC) \to M_{2n}(\bR)$)  with a real ucp  map $h : V \to M_n(\bC)$ with $h(x) = f(x) + i g(x)$ for $x \in V$.
            These identifications are another way of describing the map $\psi$ above, and its inverse. 
            
            More detailed proof: That $\psi$ is well defined and continuous is as above.  Similarly,   $\psi$ is graded because it sends certain elements of $M_{n}(CB(V, \bR)_c)$ to elements of $M_n(CB(V_c, \bC))$, and  is complex
nc affine being the restriction of a $\bC$-linear map.
 Let $f+ig \in (\ncS(V)_c)_N$ so that $c(f+ig)$  will be a ucp map from $V$ to $M_{2N}(\bR)$. The complexification of  $c(f+ig)$ is a ucp map from $V_c$ to $M_{2N}(\bC)$ given by
            \[\big(c(f+ig)_c\big)(x+iy) = \begin{bmatrix}
                f (x) & -g (x)\\  g (x) & f (x)    \end{bmatrix} + i \begin{bmatrix}  f (y) & -g (y)\\   g (y) & f (y)   \end{bmatrix}\]
            Taking the $r$th amplification shows that for $0 \leq [x_{nm} + i y_{nm}] \in M_r(V_c)$ we have 
            \begin{align*}            0 &\leq                 \big(c(f+ig)_c \big)^{(r)}            ([x_{nm} + i y_{nm}])            \\&=  
 \begin{bmatrix}  \begin{bmatrix}   f (x_{nm}) & -g (x_{nm})\\   g (x_{nm}) & f (x_{nm}) \end{bmatrix} + i \begin{bmatrix}
               f (y_{nm}) & -g (y_{nm})\\                g (y_{nm}) & f (y_{nm})            \end{bmatrix}            \end{bmatrix} .
 \end{align*} Compressing the last matrix by $\frac{1}{\sqrt{2}}\begin{bmatrix} 1_n \\ -i \cdot 1_n \end{bmatrix}$ 
gives that   \begin{align*}                0 &\leq \frac{1}{\sqrt{2}}\begin{bmatrix}                    1_n & i \cdot 1_n
\end{bmatrix}\Big(\begin{bmatrix}   f(x_{nm}) & -g(x_{nm}) \\ g(x_{nm}) & f(x_{nm})   \end{bmatrix} 
 + i  \begin{bmatrix}    f(y_{nm}) & -g(y_{nm}) \\ g(x_{nm}) & f(y_{nm})   \end{bmatrix} \Big) \frac{1}{\sqrt{2}}
\begin{bmatrix}   1_n \\ -i \cdot 1_n   \end{bmatrix} \\&= [f(x_{nm}) - g(y_{nm}) + i f(y_{nm}) + i g(x_{nm})]
            \\&= \psi(f+ig)^{(M)}([x_{nm} + i y_{nm}])            \end{align*}
            or $\psi(f+ig)$ is completely positive. It is unital because $c(f+ig)$ is unital and therefore $\psi$ sends elements of $\ncS(V)_c$ to matrix states.
                     
         The inverse of $\psi$ is $\gamma$ since for $\omega \in \ncS(K_c)$ and $x,y \in V$ we have
            \begin{align*}
                \psi(\re \omega+i \im \omega)(x+iy) 
                &= \re \omega(x)  - \im\omega(y) + i \re\omega(y) + i \im\omega(x)
                \\&= (\re\omega + i \im\omega)(x+iy)
                \\&= \omega(x+iy)
            \end{align*}
            and conversely for $f+ig \in \ncS(V)_c$ we get
            \begin{align*}
                \gamma(\psi(f+ig))
                &= \re \psi(f+ig) + i \im \psi(f+ig)
                \\&= f+ig .
            \end{align*}
            The inverse is a well defined map because $\re \omega + i \im \omega$ is in $\ncS(V)_c$. Indeed $c(\re \omega + i \im \omega)$ is ucp because for $N \in \bN$ and $0 \leq [x_{nm}] \in M_N(V)$ we have
            \begin{align*}
                c(\re \omega+i\im \omega)^{(N)}([x_{nm}])
                &= \begin{bmatrix}
                    \re(\omega(x_{nm})) & -\im(\omega(x_{nm})) \\ \im(\omega(x_{nm})) & \re(\omega(x_{nm}))
                \end{bmatrix} \in M_{2N}(\bR)
            \end{align*}
            and the latter is positive if and only if $[\re(\omega(x_{nm})) + i \im (\omega(x_{nm}))] = [\omega(x_{nm})]$ is positive. However, this is the $K$th amplification applied to $[x_{nm}]$ which is positive. Finally, the inverse is continuous because $\psi$ restricted to any $(\ncS(V)_c)_N$ is a bijective continuous map with compact domain.
             \end{proof}
             
          {\bf Remarks.} 1)\        The last result also follows from Theorem   \ref{unco}, and checking that $\ncS(V_c)$ is a reasonable complexification of  $\ncS(V).$  Indeed take 
          $Y = V^*$ and $F = (V_c)^*$ there, with $Y \subset F$ via $\psi \mapsto \psi_c$.
          
          \smallskip

          2)\ A similar proof shows that for a real operator system $V$, and cardinal $n$, on $M_n((V_c)^*) = CB(V_c , M_n(\bC))$  the canonical cone $({\mathfrak C}_c)_n$  coincides with 
the expected cone CP$(V_c , M_n(\bC))$.   Here $({\mathfrak C}_n)$ are the usual cones $({\rm CP}(V,M_{n}(\bR)))$. 
 Thus if $\varphi \in M_n((V_c)^*) \cong CB(V_c, M_n(\bC)) = CB(V,M_n(\bR))_c$, so that (uniquely) 
$\varphi = \psi_c + i \rho_c$ for $\psi, \rho \in CB(V,M_n(\bR))$, then 
$\varphi \in M_n((V_c)^*)^+ = {\rm CP}(V_c,M_n(\bC))$ if and only if $c(\varphi)  = c(\psi, \rho) \in 
M_{2n}(V^*)^+ = {\rm CP}(V,M_{2n}(\bR))$.   This also may be deduced from the statement of  Lemma \ref{ncSCommC}. 

            \begin{example}\label{Quaternions}
                Consider the quaternions $\bH$ as a real $C^*$-algebra. This example will recur 
                                repeatedly through the series of papers, so we will be brief here.
                The state space of $\bH$ is trivial, a singleton containing only the map $(a+ib+jc+kd \mapsto a)$.  However the non-commutative state space at the higher levels make up for this deficit, and is much more interesting. The complexification of $\bH$ is $M_2(\bC)$. So the first level of $\ncS(\bH_c) \cong \ncS(M_2(\bC))$ is all states on $M_2(\bC)$, which correspond to positive $2 \times 2$ trace one matrices -- i.e.\ the first level is affine isomorphic to the Bloch sphere. By Lemma \ref{ncSCommC} we have $\ncS(\bH_c) \cong \ncS(\bH)_c$ and so through complexification the first level of $\ncS(\bH)$ went from having a single element to containing a three dimensional ball's worth of elements.
            \end{example}

            Next, for $K$ real compact nc convex we consider $A(K)_c$ as a complexification of an operator system versus $A(K_c)$ as a complex operator system. Let $f,g: K \to \cM(\bR)$ be real affine maps and $x+iy \in M_N(K_c)$. Define the map $\Psi: A(K)_c \to A(K_c)$ by 
            \[\Psi(f+ig)(x+iy) = u_n^*
                \big(f(c(x+iy)) + i g(c(x+iy))\big)
                u_n . \]
            We will see this has inverse
            $\Gamma: A(K_c) \to A(K)_c$ taking $\omega \in A(K_c)$ to $\re \omega_{|K} + i \im \omega_{|K}$.

           For a real compact nc convex set $K$ there is a canonical map $\epsilon : A(K) \to A_{\bC}(K_c)$
           defined by $\epsilon(f) = f_c$, where $f_c$ is as above.            

            \begin{thm}\label{AffineCommC}
                The map $\Psi: A(K)_c \to A(K_c)$ is a ucoi with inverse $\Gamma$.
                Indeed $A(K)$ may be identified with the fixed points of the period 2 conjugate linear complete order automorphism $a_\theta$ of 
                $A(K_c)$ defined by $a_\theta(f) = \theta_{\bC} \circ f \circ \theta_K$, where $\theta_K$ is as defined after Theorem  {\rm \ref{ncrealaffinecoml}}.
            \end{thm}
            \begin{proof} For the same reason as before, and also to exhibit a complementary result, we give two proofs. Since $\theta_K$ is affine and continuous it follows that $a_\theta : A(K_c) \to A(K_c)$.  Since $\theta_K$ is  period 2, so clearly is
$a_\theta$ too.  Clearly $a_\theta$ is unital, and it is not hard to see that it is conjugate linear since $\theta_{\bC}$ is 
conjugate linear: for example if $f \in A(K_c)$ then $$a_\theta(if)(k_1 + i k_2) = \theta_{\bC} ( (if) (k_1 - i k_2))
= i a_\theta(f)(k_1 + i k_2).$$  

For $x, y \in M_n(\bR)$ we have $$\theta_{\bC}((x+iy)^*) = 
\theta_{\bC}(x^* - i y^*) = x^* + i y^*,$$ which equals $(\theta_{\bC}(x+iy))^* = (x-iy)^*$.
Thus $a_\theta(f^*) = a_\theta(f)^*$, since for example $$a_\theta(f^*)(k_1 + i k_2) = \theta_{\bC}(f^* ((k_1 - i k_2))= \theta_{\bC}(f(k_1 - i k_2)^*).$$ 
If $f \in M_n(A(K))^+ = A(K,M_n)^+$ then 
$$[a_\theta(f_{ij})(k)] = [\theta_{\bC} (f_{ij} (\theta_{K}(k)))], \qquad k \in K_c .$$  Since $\theta_{\bC}$ and its amplifications are completely positive, and $f(K) \geq 0$, we have 
$$[a_\theta(f_{ij})(k)] \geq 0,$$ so that $a_\theta$  is completely positive.

The fixed points of $a_\theta$ clearly include 
$\epsilon(A(K))$.  Indeed $a_\theta(f_c) = \theta_{\bC} \circ f_c 
\circ \theta_{K}$ is an affine extension of $f$ and so $a_\theta(f_c) = f_c$ by the uniqueness in Lemma \ref{Affineext}.
Conversely, suppose that $a_\theta(g) = g$ for $g \in A(K_c)$.
Then $\theta_{\bC}(g(\iota(k))) = g(\iota(k))$ for $k \in K$.
Thus $g(\iota(k)) \in \iota(K)$.  Let $f = Re \, g_{|K}$, a real nc affine map on $K$.  Then 
$g = f_c = \epsilon(f)$ by the uniqueness in Lemma \ref{Affineext}, since these are both nc affine extensions of $f$. 

    More detailed proof: 
                The function $\Psi$ is complex linear. We show that this map is well defined. First, $\Psi(f+ig)$ will be continuous because $f,g$ are continuous, and $\Psi(f+ig)$ is graded because $f,g$ are graded. Now, notice that for matrices $a,b \in M_N(\bR)$ we have  $$u_n^* c(a+ib)^\tran = (a+ib)^*u_n^*.$$ Similarly $c(a+ib)u_n =u_n(a+ib)$. Let $x+iy \in M_N(K_c)$ and $a,b \in M_{N,k}(\bR)$ such that $a+ib$ is an isometry, then using facts about the function $c$ in Theorem \ref{ncrealaffinecoml} and that $f,g$ are real affine gives
                \begin{align*}
                    \Psi(f+ig)&((a+ib)^*(x+iy)(a+ib))
                    \\&=u_n^*(f(c(a+ib)^\tran c(x+iy)c(a+ib)) \\&\quad+ i g(c(a+ib)^\tran c(x+iy)c(a+ib)))
                 u_n
                \\&= (a+ib)^* \big(\Psi(f+ig)(x+iy)\big)(a+ib)
               . \end{align*}
                Therefore $\Psi(f+ig)$ preserves compressions. A similar proof shows that it preserves direct sums and therefore $\Psi(f+ig)$ is affine. If $f+ig$ is positive in $A(K)_c$ then $c(f+ig)$ is positive in $M_2(A(K))$, or for any $k \in K$ we have $c(f+ig)(k) \geq 0$. Compressing this matrix by $u_n$ gives that $f(k)+ig(k) \geq 0$ for all $k \in K$. From this we see that for any $x+iy \in K_c$ we have $\Psi(f+ig)(x+iy) \geq 0$ as we are just taking the adjoint of a positive matrix. Therefore, $\Psi$ is positive and a similar proof shows that our map is completely positive. The unit of $A(K)_c$ is $1+i0$ where $1$ is the constant function on $K$. Note $\Psi(1+i0)(x+iy) = 1_N$ and so $\Psi$ is unital.

                To show that $\Psi$ has inverse $\Gamma$ we see that for bounded $\omega: K_c \to \cM(\bC)$ and $x+iy \in (K_c)_n$ we have
                \begin{align*}
                    \omega(x+iy) 
                    &= \omega(u_n^* c(x+iy) u_n)
                \\&=   u_n^*\omega( c(x+iy))  u_n
                \\&= \Psi((\re \circ \omega_{|K})+ i (\im \circ \omega_{|K}))(x+iy) = \Psi(\Gamma (\omega))(x+iy) .
                \end{align*}
                Conversely for $f+ig \in A(K)_c$ we have
                \begin{align*}
                    \Gamma(\Psi(f+ig))
                    &= \re \Psi(f+ig)_{|K} + i \im \Psi(f+ig)_{|K}
                    \\&= f+ig
                \end{align*}
                where the last equality comes for instance from the fact that 
                \begin{align*}
                    \re \Psi(f+ig)(x) 
                    &= \re \frac{1}{\sqrt{2}} \begin{bmatrix}
                    1_n & i \cdot 1_n
                \end{bmatrix}
                \big(f(x \oplus x) + i g(x \oplus x)\big)
                \frac{1}{\sqrt{2}} \begin{bmatrix}
                1_n \\ -i \cdot 1_n
            \end{bmatrix}
            \\&= \re (f(x) + ig(x)) = f(x) .
                \end{align*}
                Finally, we need to show that $\Gamma$ is ucp. It is unital because $\Psi$ is unital. If $[\omega_{nm}] \in M_N(A(K_c))$ is positive then it maps all $x+iy \in K_c$ to positive matrices. So, for $x \in K$ we have 
                
                \[c([\re \circ \omega_{nm} + i\im \circ \omega_{nm}])(x) = c([\re(\omega_{nm}(x)) + i \im (\omega_{nm}(x))]) = c([\omega_{nm}(x)])\]
                Now $[\omega_{nm}(x)]$ is positive in $M_N(\bC)$ and so $c([\omega_{nm}(x)]$ is positive in $M_{2N}(\bR)$. Therefore, $c([\Gamma(\omega_{nm})])$ is a positive element of $M_2(A(K))$ which by definition means $[\Gamma(\omega_{nm})]$ is positive in $A(K)_c$.            \end{proof}

        \section{Real Categorical Duality}\label{SecCatDuality}

            For any real compact nc convex set $K$, $A(K)$ will be a real operator system by Section
             \ref{AffAreOpSys}. On the other hand, given a real operator system $V$, $\ncS(V)$ will be a real compact nc convex set. We have the following duality extending the complex case in Theorem 3.2.2 of \cite{DK}:
            
            \begin{theorem}\label{AffineCateg}
                Let K be a (real or complex) compact nc convex set, then $K \cong \ncS(A(K))$ via the complex affine homeomorphism $\Lambda: K \to \ncS(A(K))$ where
                \[\Lambda(x)(\varphi)=\varphi(x)\]
                for $x \in K$ and $\varphi \in A(K)$.
            \end{theorem}
            
            Conversely we have (extending Theorem 3.2.3 in \cite{DK}):
            
            \begin{theorem}\label{ncSCateg}
                Let $V$ be a closed (real or complex) operator system. For $v \in V$ define the function $\hat{v}: \ncS(V) \to \cM(\bF)$ by
                \[\hat{v}(\varphi) = \varphi(v)\]
                for $\varphi \in \ncS(V)$. This map is a continuous nc affine function. The map $\hat{\phantom{v}}: V \to A(\ncS(V))$ taking $v$ to $\hat{v}$ is a ucoi.
            \end{theorem}

            We first prove Theorem \ref{AffineCateg} in the real case.  For a real compact nc convex set, we have that $K_c$ is a complex nc convex set and therefore isomorphic to $\ncS(A(K_c))$ which in turn is isomorphic to $\ncS(A(K))_c$. We have the embedding $\iota$ of our real nc convex sets into their complexification, so we just need to make sure $\iota(K)$ is mapped onto $\iota(\ncS(A(K))$ through the above isomorphisms. Or, the following diagram commutes
            \[
            \begin{tikzcd}
            K_c \arrow[r, "\Lambda"] & \ncS(A(K_c)) \arrow[r, "\Psi^*"] & \ncS(A(K)_c) \arrow[r, leftarrow, "\psi"] & \ncS(A(K))_c \\
            K \arrow[u, hook, "\iota"] \arrow[rrr, "\Lambda"] & & & \ncS(A(K)) \arrow[u, hook', "\iota"]
            \end{tikzcd}
            \]
            Here, $\Psi^*: \ncS(A(K_c))$ is defined by $\Psi^*(f)(x+iy) = f(\Psi(x+iy))$. So now we just diagram chase. Let $k \in K_N$, then going to the right, for $f+ig \in A(K)_c$ we have
            \begin{align*}
                \psi(\Lambda(k)+i0)(f+ig)
                &= \Lambda(k)(f) + i \Lambda(k)(g) = f(k) + ig(k) .
            \end{align*}
            Going up we have:
            \begin{align*}
                \Psi^*(\Lambda(k+i0))(f+ig)
                &= \Lambda(k+i0)(\Psi(f+ig))
                \\&= \Psi(f+ig)(k+i0)
                \\&=  \frac{1}{2} \begin{bmatrix}
                    1_N & i \cdot 1_N
                \end{bmatrix}
                \big(f(c(k+i0)) + i g(c(k+i0))\big) \begin{bmatrix}
                1_N \\ -i \cdot 1_N
            \end{bmatrix}
            \\&= \frac{1}{2} \begin{bmatrix}
                    1_N & i \cdot 1_N
                \end{bmatrix}
                \big(f(k \oplus k) + i g(k \oplus k)\big) \begin{bmatrix}
                1_N \\ -i \cdot 1_N
            \end{bmatrix}
            \\&= \frac{1}{2} \begin{bmatrix}
                    1_N & i \cdot 1_N
                \end{bmatrix}
                \big(f(k) \oplus f(k) + i g(k) \oplus g(k)\big) \begin{bmatrix}
                1_N \\ -i \cdot 1_N
            \end{bmatrix}
            \\&= f(k) + ig(k) .
            \end{align*}
        Therefore, starting at the bottom left and going clockwise, is the same as going right and then anticlockwise in the diagram.  That is, the diagram commutes,   resulting in a nc affine homeomorphism between $K$ and $\ncS(A(K))$.
        For the reader's convenience we check surjectivity. 
        If $\varphi \in \ncS(A(K))$ then $\iota(\varphi) \in \ncS(A(K))_c = 
            \ncS_{\bC}(A_{\bC}(K_c))$.  (Indeed 
            $\varphi_c$ is a matrix state of 
            $A(K)_c$, and thus gives a matrix state of 
            $A_{\bC}(K_c)$ by composition with the canonical map 
            $A_{\bC}(K_c) \to A(K)_c$.)  Thus, by Theorem 3.2.2 of \cite{DK} (or Theorem 16.6.4 of \cite{Dav})
            there exists $x + iy \in K_c$ 
                        mapping to $\iota(\varphi)$.  That is, for all
            $f \in A_{\bC}(K_c)$ we have 
            $$f(x + iy) = \iota(\varphi)(f) = \varphi({\rm Re} f_{|K}) + i \varphi({\rm Im} f_{|K}).$$  Here ${\rm Re}, {\rm Im}$ here are the real affine functions coming from Lemma \ref{reisAffine}. 
            In particular, replacing $f$ by $\epsilon(f) = f_c$ for $f \in A(K)$,  taking real parts,
            and remembering that $x \in K_n$, we have
            $f(x) = {\rm Re} (f(x) + i f(y)) = \varphi(f)$ for $f\in A(K).$
            That is $\Lambda(x) = \varphi$, so that $\Lambda$ is surjective.  

\bigskip

        {\bf Remark.} Alternatively, the
            the proof of Theorem 3.2.2 of \cite{DK} works in the real case.  
            
                      \bigskip

        The proof of Theorem \ref{ncSCateg} in the real case has a similar idea (although the proof in  \cite{DK,Dav} does not 
        work in the real case).  We want to show the following diagram commutes
        \[
            \begin{tikzcd}
            V_c \arrow[r, "\wedge"] & A(\ncS(V_c)) \arrow[r, "\psi^*"] & A(\ncS(V)_c) \arrow[r, leftarrow, "\Psi"] & A(\ncS(V))_c \\
            V \arrow[u, hook, "\iota"] \arrow[rrr, "\wedge"] & & & A(\ncS(V)) \arrow[u, hook', "\iota"]
            \end{tikzcd}
            \] 
        To show the diagram commutes, let $v \in V$. Going to the right we have
        \begin{align*}
            \Psi(\hat{v}+i0)(x+iy)
            &= \frac{1}{2} \begin{bmatrix}
                    1_n & i \cdot 1_n
                \end{bmatrix}
                \big(\hat{v}(c(x+iy)\big) \begin{bmatrix}
                1_n \\ -i \cdot 1_n
                \end{bmatrix}
            \\&= \frac{1}{2} \begin{bmatrix} 
                    1_n & i \cdot 1_n 
                \end{bmatrix} 
                \big(c(x(v)+iy(v))\big) \begin{bmatrix} 
                1_n \\ -i \cdot 1_n 
                \end{bmatrix} 
            \\&= x(v) + iy(v) 
        \end{align*} 
        and going up we get 
        \begin{align*} 
            \psi^*(\widehat{v+i0})(x+iy) 
            &= (\widehat{v+i0})(\psi(x+iy)) 
            \\&= \psi(x+iy)(v+i0) 
            \\&= x(v) + iy(v) 
        \end{align*} 
        To show the surjectivity of $\wedge$ from the diagram, first take $f+i0 \in A(\ncS(V))_c$. Via the diagram, $f+i0$ corresponds to evaluation at some $x+iy \in V_c$. Via the canonical evaluation map and $\psi^*$ we have that
    \[\psi^*(\widehat{x+iy})(\varphi + i \theta) = \varphi(x) - \theta(y) + i \varphi(y) + i \theta(x).\]
    Applying $\Gamma$ to this gives
    $(\varphi \mapsto \varphi(x)) + i (\varphi \mapsto \varphi(y)) \in A(\ncS(V))_c$.
    This must equal $f+i0 \in A(\ncS(V))_c$.
    This forces $f = \hat{x}$ and $y = 0$.   Indeed for any element $\varphi \in \ncS(V)$ we have $\varphi(y) = 0$. This implies $y = 0$ by the fact in Example 2.2 (or e.g.\ \cite[Corollary 3.2]{BR}). This shows surjectivity.

        It follows as in \cite{DK,Dav} that  the categories NCConv$_{\bR}$ (real compact nc convex sets and continuous affine nc maps) and OpSy$_{\bR}$ (real operator systems
        and ucp maps) are dually equivalent via a contravariant functor.  Thus for example 
        $f \mapsto f \circ \tau :  A(K_2) \to A(K_1)$, composition of $f \in A(K_2)$ with a continuous affine nc map $\tau : K_1 \to K_2$, is ucp.  For compact nc convex sets $K$ and $L$ we have  $A(K)$  and $A(L)$ unitally complete order isomorphic if and only if $K$ and $L$ are affinely homeomorphic. Hence two operator systems are isomorphic if and only if their nc state spaces are affinely homeomorphic.

\smallskip

 {\bf Remark.} If we view a complex nc compact convex $K$ as a real nc compact convex $K_r,$ then $K_r$ corresponds by this duality to a real operator system $W$.    Can one describe $W$ simply in terms of the  complex operator system $V$ associated with $K$? 

\smallskip

        We now characterize the real compact nc convex sets which correspond to real operator systems which are the selfadjoint part of another operator system.

        \begin{cor} \label{coti} A real compact  nc set $K$  corresponds under the duality above to a real operator system with trivial involution if and only if every $k \in K$ is symmetric (that is, $k = k^\tran$). 
            \end{cor} 

            \begin{proof}  For ucp $\varphi : V \to M_n$ we have $$\varphi(x)^\tran = \varphi(x^*) \qquad x \in V.$$  Thus if $V$ has trivial involution then 
            $K = \ncS(V)$ is symmetric. 
            Similarly, if the latter is symmetric 
            then $\varphi(x^*) = \varphi(x)$ for all 
            such $\varphi$ and $x \in V$, so that $x = x^*$. 
                \end{proof} 
                
                                We shall call real compact  nc sets satisfying the condition in the last result {\em symmetric}.

If  $K$ is a real nc convex set then the complexification $K_c$ of $K$ consists of elements $z= x + iy$ where $c(x,y) \in K$.   We call the set of such 
$x$ (resp.\ $y$) the {\em real part} (resp.\ {\em imaginary part}) of $K_c$.  These are also nc convex sets. 
The imaginary part of $K_c$ is a bit mysterious, and the following discussion is intended to clarify part of the mystery. 
If $K$ is also compact then we may assume by the duality above  that 
$K$ is the nc state space ${\rm ncS}(V)$ of a real  operator system $V$.  That is, 
$K_n = {\rm UCP}(V,M_n(\bR))$.  Let $W = B(H) = M_n(\bR)$ for a cardinal $n$.  Let $C$ be the 
real nc convex set with $C_n = {\rm CP}(V,M_n(\bR))$. 
Write $D = \{ u \in {\rm CB}(V,M_n(\bR)) : u^* = -u , u(1) = 0 \}$,  with its canonical operator space structure, 
and let $\bB_0$ be the set of matrix unit balls for $D$ (see Example \ref{Matball}), a real nc convex compact set.  
We mention a result from \cite{BP}: 

\begin{theorem} \label{coco} {\rm \cite[Theorem 4.2]{BP} } We have that $\bB_0$ is the `imaginary part' 
of the complexification $K_c$ of the real nc compact convex set $K$ above. 
Also, $D$ is the imaginary part of the complexification $C_c$ of the real nc convex set $C$ above. 
\end{theorem} 

Of course the `real parts' of these nc sets are respectively $K$ and $C$. 
  It seems remarkable that every 
map in $\bB(D)$ is the imaginary part of a ucp map.

        \section{Bipolar Theorem}\label{SecBipolarThrm}

        \begin{subsection}{Real Bipolar Theorem}

            Let $E$ be a real dual operator space and $K$ a real nc set over $E$. The polar of $K$ is a real nc convex set over $E^*$ defined by 
            \[K_n^\circ = \{ \varphi \in M_n(E^*):  \Rea \, \varphi^{(m)}(v) \leq 1_{nm} \text{ for all } v \in K_m, m \leq \kappa\}\]
            The set $K^\circ = \bigsqcup_n K_n^\circ$ is a closed real nc convex set. This definition is the same as in the complex case, except that Effros and Winkler only consider finite $n, m$ in their definition.  Note though that  if 
            $\Rea \, \varphi^{(m)}(v) \leq 1_{nm}$ for all finite $m$ and $v \in K_m,$ then it is easy to argue that 
            we have the same relation
            for infinite $m$.
            Thus we may take $m < \infty$ in the definition above.  Hence the definition makes sense
            and produces a closed real nc convex set even if $K$ is only a matrix convex set.
            
            \begin{prop}
                Let $K$ be a real nc set (or real matrix set) in the real dual operator space $E$. Then $(K^\circ)_c \cong (K_c)^\circ$ via the same maps as in Lemma {\rm \ref{ncSCommC}}.
            \end{prop}
            \begin{proof} Of course 
                $K_c$ is a complex nc convex set containing $0_{E_c} \in E_c$. 
                Let $\gamma: (K_c)^\circ \to (K^\circ)_c$ taking $\omega$ to $\re \omega_{|E} + i \im \omega_{|E}$, which will have inverse $\psi$.   To see that $\gamma$ maps into $(K^\circ)_c$, let  $\omega \in (K_c)^\circ_n$ and $v \in K_m$.  We have 
                \begin{align*}
                    \Rea \, c(\gamma(\omega))^{(m)}(v) 
                    &= \Rea \, c(\re \omega^{(m)}(v+i0) + i \im \omega^{(m)}(v+i0))
                    \\&= \Rea \, c(\omega^{(m)}(v+i0))
                    \\&= c(\Rea \, \omega^{(m)}(v+i0)).
                \end{align*}
                We also have
                \[c(\Rea \, \omega^{(m)}(v+i0)) \leq 1_{m\cdot2n} \iff \Rea \, \omega^{(m)}(v+i0) \leq 1_{mn} .\]
                However if $v \in K$ then $v+i0 \in K_c$, and so this is true. So $\gamma$ maps into $(K^\circ)_c$.  We saw in Lemma $\ref{ncSCommC}$ that $\gamma$ is a bijective continuous affine map with  continuous inverse $\psi$.  Hence we will be done if   $\psi((K^\circ)_c) \subseteq (K^\circ)_c$.

                The inverse map $\psi: (K^{\circ})_c \to (K_c)^\circ$ takes $f+ig \in ((K^{\circ})_c)_N$ and $x+iy 
                \in E_c$ to  $$\psi(f+ig)(x+iy) = f(x) - g(y) + i f(y) + i g(x).$$ To show this indeed maps into $(K_c)^\circ$, suppose that $x+iy = [x_{nm} + i y_{nm}] \in (K_c)_M$,  so that $c(x+iy) \in K_{2M}$. We want to show that $$\Rea \, [f(x_{nm}) - g(y_{nm}) + i f(y_{nm}) + i g(x_{nm})] \leq 1_{NM}.$$ However, as shown in Lemma \ref{ncSCommC} we have
                \begin{align*}
                \Rea \, [&f(x_{nm}) - g(y_{nm}) + i f(y_{nm}) + i g(x_{nm})] 
                \\&= \Rea \, u_N^*\Big(\begin{bmatrix}
                f(x_{nm}) & -g(x_{nm}) \\ g(x_{nm}) & f(x_{nm})
                \end{bmatrix} 
                + i 
                \begin{bmatrix}
                    f(y_{nm}) & -g(y_{nm}) \\ g(y_{nm}) & f(y_{nm})
                \end{bmatrix} \Big) u_N
                \\&= \Rea \,u_N^*\big(c(f+ig)^{(M)}(x)+ic(f+ig)^{(M)}(y)\big) u_N\\
                &= u_N^* \Rea \, \big(c(f+ig)^{(M)}(x)+ic(f+ig)^{(M)}(y)\big) u_N .
                \end{align*}
                Because $u_n$ is an isometry we see that the latter is $\leq 1_{NM}$
                             as for $f+ig \in (K^\circ)_c$ and $x+iy \in K_c$ we have that
                \begin{align*}
                \Rea \, c(f+ig)^{(2N)}(c(x+iy)) = 
                    \Rea \, \begin{bmatrix}
                        c(f+ig)^{(N)}(x) & - c(f+ig)^{(N)}(y)\\
                        c(f+ig)^{(N)}(y) & c(f+ig)^{(N)}(x)
                    \end{bmatrix}  \leq 1_{4NM} .
                \end{align*}
                This completes the proof.            \end{proof}

            \begin{theorem}[Bipolar Theorem] \label{bipo} 
                Let $K\subseteq E$ be a closed real or 
                                complex nc convex set containing $0_E \in E$. Then $K^{\circ \circ } \cong K$.
            \end{theorem}
            \begin{proof}  
            Clearly  $K \subseteq K^{\circ \circ}$.
            If $x \in (K^{\circ \circ})_n \setminus K_n$ then by Theorem \ref{separationthrm} there exists  a normal completely bounded map $\varphi: E \to M_n(\bF)$ such that $\Rea \, \varphi_n(x) \not \leq  1_n \otimes 1_n$ but for all $p$ and $k \in K_p$ we have $\Rea \, \varphi_p(k) \leq 1_p \otimes 1_n$.  
            Then $\varphi \in (K^\circ)_n$, and we obtain the contradiction 
            $\varphi_n(x) \leq  1_n \otimes 1_n$. So 
                $K = K^{\circ \circ}$.            \end{proof}

        Effros and Winkler's application of  the bipolar theorem in \cite[Section 5]{EW} essentially works for us too.  That is a weakly compact nc convex set $L$ with $L_1=K$ is sandwiched between  minimal and maximal nc convex sets which are $K$ at level 1.   Effros and Winkler write these as
        $\hat{L}$ and $\check{L}$. 
        We recall their construction.   We will assume that $L$ is a weak* closed (or weak* compact) real or complex nc convex set in a dual operator space $E$.   By translating if necessary by a fixed vector in $K$ 
        one may assume without loss that $0 \in L$.    We then define the desired minimal nc convex set $\hat{L}$  to be the closed convex hull $C$ in $E$ of 
   $L_1=K$.   The    `maximal one' is defined as the prepolar of  ``the `minimal one'  of the polar of $L$''.
   That is, $\check{L} = \widehat{L^\circ}$.  If we had previously translated to ensure that $0 \in L$, we must then apply the inverse of the above 
   translation to these minimal and maximal nc sets.        
   
   By the Bipolar Theorem \ref{bipo} as in \cite{EW} we see that 
   $\hat{L}$ and $\check{L}$ are respectively the smallest and largest closed nc convex set $D$ in $E$ with $D_1 = K$. 
          In the real case, though, this is sometimes more helpful under the restriction that the compact nc convex set corresponds to an operator system with trivial involution, as we will explain at the end of the next section.    Under this last restriction $\hat{L}$ agrees with a familiar nc set.   On the other hand,
    suppose that $L$ is     the noncommutative state space of the quaternions or of 
       the real $C^*$-algebra $\bC_r$ (these systems have nontrivial involution). 
       The reader can   check in these cases that $\hat{L}$ is a trivial compact nc convex set (a translate of 
   $(0)$), which  conveys little  useful information. 
   
         \end{subsection}

        \begin{section}{Max and Min nc convex sets}\label{SecMinMax}

        By a real function system (in Kadison's original sense)
        we mean (concretely) a unital subspace $\cS$ of $C_{\bR}(K)$ for a compact Hausdorff $K$, or abstractly (via Kadison's theorem), a (real) ordered vector space 
        $V$ with an archimedean order unit.  Note that no involution is mentioned here, indeed the canonical involution on $\cS$ is trivial (the identity map).
        We may view $V$ as having this trivial invoution.         Similarly a complex function system is (concretely) a unital selfadjoint subspace of $C_{\bC}(K)$ for compact  $K$, or abstractly 
                a complex $*$-vector space which is ordered
        (i.e.\ with a proper selfadjoint cone $E^+
                \subset E_{\rm sa}$), and possesses an archimedean order unit for $E_{\rm sa}$.
        The complex $*$-vector space version of Kadison's         characterization of Archimidean order unit spaces as function systems follows immediately from the real         case.
                (See e.g.\ Lemma 1.2 and the lines after it in the last paper mentioned in the Acknowledgements.  See also Theorem 1.1 and Lemma 1.2 in the other paper mentioned there.)  
                
        These form categories, with the morphisms being unital selfadjoint positive maps,
        or equivalently (by basic results in e.g.\ \cite[Section 2]{Pnbook} or \cite[Section II.1]{Alfsen}) unital selfadjoint contractions.  

     \begin{prop}\label{omincat}  
        The category of complex function systems is equivalent  to the category of real function systems.  Moreover every real function system has a unique reasonable  function system complexification, and every complex function system has unique real structure, that is, is the complexification of an essentially unique real function system. \end{prop}

            \begin{proof}    For real valued functions $f, g$ on a set $K$ we have $$\| f + ig \| = \sup \{ \sqrt{f(x)^2 + g(x)^2} : x \in K \} = \sup \{ \| s f + t g \|_{C(K,{\mathbb R})} : s^2 + t^2 \leq 1 \}.$$  Here $s, t$ are real.  It follows that the `function system complexification' of a 
 real function system $S$ is unique and  reasonable (we  assume that 
 the embedding of $S$ into the function system complexification is as real valued functions).  This complexification is the so-called Taylor complexification (see e.g.\ \cite{MMPS}). 

 Conversely, every complex function system  $V$ is a reasonable complexification of a real function system.  Indeed suppose that $V$ is a unital selfadjoint subspace of $C(K,\bC)$.
 Then $V_{\rm sa}$ is a real function system with the inherited cone (for example 
 it is clearly an ordered real space with archimedean order unit $1_V$).  
  Moreover $V$ is a reasonable complexification of $V_{\rm sa}$.

Clearly the above defines two functors between the categories. Notice that 
a unital positive map $T : V \to W$ in the complex category is completely positive, and selfadjoint, so $T(V_{\rm sa}) \subseteq W_{\rm sa}$. 
Thus it is clear that the category of complex function systems is isomorphic to the category of real function systems.

Any operator system $S$ whose complexification is $V$, is the set of fixed points for some  period 2 conjugate linear unital order 
isomorphism $u : V \to V$.
Note that $u$ is selfadjoint since it is positive.  Let $w = u_{|V_{\rm sa}}$, which is a period 2 order automorphism of $V_{\rm sa}$.   Then $u \circ w_c$ is a period 2 conjugate linear  unital order isomorphism $V \to V$ whose fixed points are exactly $V_{\rm sa}$.  That is, $u(w_c(v)) = \bar{v}$ for $v \in V$.  Thus up to the unital order isomorphism $w_c$,  the complexification of $S$ can be identified with the  complexification of $V_{\rm sa}$. \end{proof}

{\bf Remark.}  Of course the analogues for operator systems  of most of the assertions in the last result are (badly) false in general.  
In particular a complex operator system  $V$ need not be a reasonable complexification of  $V_{\rm sa}$.  For example, $M_2(\bC)$ is 
not a reasonable complexification of  $M_2(\bC)_{\rm sa}$ (see the discussion after Lemma \ref{reisAffine}).  
Lemma \ref{ccl} however is a partial result along these lines.
We  also always have a canonical one-to-one real continuous nc affine map ${\rm ncS}_{\bC}(V) \to {\rm ncS}_{\bR}(V_{\rm sa})$, taking 
$\varphi \mapsto \re \, \varphi$.
Since $\re$ is completely contractive this map is well defined.  It is one-to-one since Re $\varphi = 0$ implies that Re $(i \varphi) = 0$.   It is surjective if $V = {\rm OMAX}(V)$ for example. 

\medskip

We now consider the nc convex sets canonically associated with the function system  and its complexification.

        Let $K$ be a classical compact convex set in a real dual Banach space $E$.   Since the beginnings of the subject of matrix and nc convexity, authors have shown that in the complex case there is a smallest and largest matrix/nc convex set which agrees at the first level with $K$ (see e.g.\ \cite[Section 5]{EW}, or \cite[Section 1.2.3]{PSS} and references therein).  It seems to us that in general
         these sets depend on the particular embedding of $K$ into a LCTVS operator space.   In this paper 
        we will define, in both the real and complex case, Min$(K)$ to be the closed nc convex hull in 
        Max$(E)$ of $(K, \emptyset,  \emptyset, \cdots)$.   (For the definition of Min and Max of operator spaces
        and their properties see e.g.\ \cite{BLM,Sharma}). This is the smallest compact nc convex set containing $(K, \emptyset,  \emptyset, \cdots)$.  
We remark that our notation conflicts with that in \cite[Section 5]{KS}, who call this max$(K)$ perhaps 
because they want to regard it as the largest compact nc convex in their ordering.

         \begin{lemma}\label{ominw}  Let $K$ be a classical compact convex set as above. 
        At the first level ${\rm Min}(K)$ is simply $K$, at the $n$-th level it (that is, $({\rm Min}(K)_n$) is the weak* closure in 
        $M_n(E)$ of the ordinary convex hull ${\rm co}(C)$ 
        of the set $C$ of terms $a \otimes k$ for a (trace 1 positive selfadjoint) density matrix $a \in M_n(\bF)^+$ and $x \in K$. 
            \end{lemma}

            \begin{proof} To see this first note that $${\rm co}(C) = \{ \sum_{j=1}^m \, v_j^* k_j v_j : m \in \bN, k_j \in K, v_j \in M_{1,n}  , \sum_j \, v_j^*  v_j = I_n \}.$$
Indeed if $x$ is in this weak* closure $W_n$ in $M_n(E)$, a limit of $x_t$, where $x_t \in {\rm co}(C)$, and $\beta \in M_{m,n}$ is an isometry, then $\beta^* x_t \beta \to \beta^* x \beta$ weak*.
Thus these weak* closures $(W_n)$  satisfy (3) in the definition of a nc convex set, and similarly it 
satisfies (2) there, if we also use the fact that $\sum_i \, \alpha_i x_i \alpha_i^*$ converges weak*.   Hence $(W_n)$ 
is a nc convex set.  It is closed, and hence compact. 
Indeed we may assume that $K \subseteq {\rm Ball}(E)$, and hence 
in $M_n(E)$ the nc convex combinations are in 
${\rm Ball}(M_n({\rm Max}(E)))$.  Since the latter ball is weak* compact, so is $W$. 
Clearly then this is the smallest closed nc convex set which agrees at the first level with $K$. \end{proof}

We remark that the weak* closure is unnecessary if $A(K)$ is finite dimensional and $n < \infty$. For the convex hull of a compact set in a finite dimensional LCTVS is compact, and the set $C$ above is easily seen to be compact in $M_n(E)$. 

If $K$ is the classical state space 
of an operator system and $n$ is a cardinal then the weak* closed convex hull $W_n$ of the set $C$ 
 defined in Lemma \ref{ominw}, may be called 
 the {\em separable} (i.e.\ {\em nonentangled}) matrix states of  $V$.
 The lemma asserts that these 
 nonentangled states (for all levels $n$, i.e.\ $W = (W_n)$) are the elements of the nc convex set Min$(K)$. 

We shall see next that 
this is also exactly  the nc/matrix state space of OMIN$(A(K))$. 

       \begin{lemma}\label{CEu} Let $\varphi : V \to B(H)$ be a completely positive selfadjoint map on a real operator system.  Let $a = \varphi(1)^{\frac{1}{2}}$.  
 Then 
there exists ucp  $\Psi : V \to B(H)$ such that $\varphi = a \Psi(\cdot) a$.  \end{lemma}

            \begin{proof} Note that
          $\varphi : V_c \to B(H)_c$
            is completely positive, so by e.g.\ \cite[Lemma 2.2]{CEinj}  there exists ucp  $\Psi : V_c  \to B(H)_c$
           with $\varphi_c = a \Psi(\cdot) a$.  Inspecting the proof of the last cited result we 
see that $\Psi(V) \subseteq B(H)$ and $\varphi = a \Psi(\cdot) a$. \end{proof}
        
        \begin{prop}\label{omin}  Let $K$ be a classical compact convex set. Then $A({\rm Min}(K)) = {\rm OMIN}(A(K))$.
        That is, ${\rm Min}(K)$ is the nc convex set corresponding to 
${\rm OMIN}(A(K))$ via the functorial correspondence between compact nc convex sets and operator systems. In particular, ${\rm Min}(K)$ consists of the ucp maps 
${\rm OMIN}(A(K)) \to M_n$, for all $n \leq \kappa$. 
            \end{prop}

            \begin{proof}  We just prove the real case. To prove this note that if $f \in A({\rm Min}(K))$ then clearly 
$f \in A(K)$.  Conversely, suppose that $f \in A(K)$.  We may assume that $K$ is a subset of a real dual Banach space such that $f$ extends to a linear continuous $\varphi \in E^*$.  Indeed we can take  $E = {\rm Max}(A(K)^*)$.   The restriction of $\varphi_n$  to $K_n$ defines the desired function from Min$(K)_n$ to $M_n$ extending $f$.  Call this function $f_n$, then 
$f_n(R \otimes k) = R f(k)$ for density matrix $R \in M_n^+$ and $k \in K$.  This defines (the unique) nc affine function $\hat{f}$ in $A({\rm Min}(K))$, which at first level is  $f : K \to \bF$. 

Note that $\| \hat{f}_n \| = \| f \|$. 
To see this note that for $x = \sum_j \, v_j^* k_j v_j$
with $k_j \in K, v_j \in M_{1,n}$ and $\sum_j \, v_j^*  v_j = I_n$ we have 
$$\| \hat{f}_n(x) \| = \| \sum_j \, v_j^* f(k_j) v_j \|
\leq \max_j \, |f(k_j)| \leq \| f \|.$$ By density and continuity we see that  $\| \hat{f}_n(x) \| \leq \| f \|$ for all
$x \in ({\rm Min}(K))_n$. 

This proves that the canonical map $A({\rm Min}(K)) \to {\rm OMIN}(A(K))$ is a unital isometry, and hence by a basic property of OMIN it is a  completely positive complete contraction.    Conversely, suppose that 
$f = [f_{ij}] \in M_n({\rm OMIN}(A(K)))^+$, so that
$f(k) \geq 0$ for all $k \in K$.  Claim: $[\widehat{f_{ij}}]
\in M_n(A({\rm Min}(K)))^+$.  Indeed for $x$ as in the last paragraph, we have for an appropriate scalar matrix $V$ that 
$$[\widehat{f_{ij}}(x)] = [\sum_j \, v_j^* f_{ij}(k_j) v_j] = V^* {\rm diag} ( f(k_1) , \cdots , f(k_n) ) V \geq 0 .$$
Thus $[\widehat{f_{ij}}] \geq 0$, by density and continuity. This proves the Claim, so that the canonical map $A({\rm Min}(K)) \to {\rm OMIN}(A(K))$ is a unital complete order isomorphism, hence also a complete isometry. 
\end{proof} 

{\bf Remark.}  If $K$ is a classical compact convex set {\em regularly embedded} in a dual operator space $E$ (see e.g.\ II.2 in \cite{Alfsen}
 or the last paper mentioned in the Acknowledgements below), then one may define a
variant  of Min$(K)$ as the closed nc convex hull of $K$ in $E$.   Since it is regularly embedded any $f \in A(K)$ extends to weak* continuous  $\varphi \in E^*$.
 The last proof then works to show that this nc convex set is topologically affine nc isomorphic to 
the nc convex set corresponding to 
${\rm OMIN}(A(K))$.  Hence also to Min$(K)$ as defined earlier (above Lemma \ref{ominw}). 

\bigskip

Most of the last result in the complex case and for $n < \infty$ can also be deduced from an assertion in 
\cite[Theorem 4.8 and Remark 4.5]{PTT}, and is equivalent to that assertion. 
Indeed we use the above to give a generalization of  this result:

\begin{cor} \label{pttth} Let $V$ be a 
real or complex function system, $H$ a real or complex Hilbert space, and $n < \infty$.  Any element of $M_n({\rm OMIN}(V)^d)^+$, or any completely positive selfadjoint
map ${\rm OMIN}(V) \to B(H)$, is a point weak* limit of a uniformly bounded 
net of maps of the form 
$\sum_j \, v_j^* \varphi_j v_j = \sum_j \, (v_j^*  v_j) \otimes \varphi_j$, for (scalar valued) states $\varphi_j$ on $V$ and 
row vectors $v_i$ with real or complex entries and $\sum_j \, v_j^* v_j$ strongly convergent.
\end{cor}
            
            \begin{proof}  Let $\varphi : V \to B(H)$ be completely positive and selfadjoint.  Let $a = \varphi(1)^{\frac{1}{2}}$. By Lemma \ref{CEu} 
there is a ucp $\Psi : V \to B(H)$ such that  $\varphi = a \Psi(\cdot) a$.  By Lemma \ref{ominw} and its proof we have that $\Psi$ is a point weak* limit of maps of the form
$\sum_i \, v_i^* \varphi_j v_j$ for states $\varphi_j$ and with $\sum_i \, v_i^*  v_j = 1$. 
Hence $\varphi$ is a point weak* limit of maps of the form
$\sum_j \, a v_j^* \, \varphi_j(\cdot) \,  v_j a$.  Note that $\| \sum_j \, a v_j^*  v_j a
\| = \| \varphi(1) \|$, so the net is uniformly bounded by $\| \varphi(1) \|$.
\end{proof} 

              \begin{lemma} \label{omin2} Let $K$ be a classical compact convex set (in a real dual Banach space). Then 
  ${\rm Min}_{\bR}(K)_c = {\rm Min}_{\bC}(K)$ as nc convex sets, with both equaling $K$ at the first level.
    \end{lemma}
            
            \begin{proof}       Complexifying the relation $A_{\bR}({\rm Min}_{\bR}(K)) =   {\rm OMIN}_{\bR}(A(K)),$  we have
$$A_{\bC}(({\rm Min}_{\bR}(K))_c) =
{\rm OMIN}_{\bR}(A(K))_c = 
{\rm OMIN}_{\bC}(A_{\bC}(K)) .$$
In the last equality we used \cite[Proposition 9.19]{BR} taking $V = {\rm OMIN}(A_{\bR}(K))$, so that $V_c = A_{\bC}(K)$
(note that $A_{\bR}(K)$ has the trivial involution at the first level). 
However ${\rm OMIN}_{\bC}(A_{\bC}(K)) = A_{\bC}({\rm Min}_{\bC}(K))$ by the complex case of Proposition  \ref{omin}. 
It follows by the functorial correspondence between compact nc convex sets and operator systems that  ${\rm Min}_{\bR}(K)_c = {\rm Min}_{\bC}(K)$. 
        \end{proof}

An OMIN (resp.\ OMAX) operator system is just a (real or complex) function system with  OMIN (resp.\ OMAX) operator system structure, or equivalently equals OMIN$(A(K))$ 
(resp.\ OMAX$(A(K))$) for a classical
closed convex  set $K$.  As discovered in \cite[Section 9]{BR} (see the lines preceding our
Subsection  \ref{SecBasicDefinitions} above), 
a real  AOU $*$-space
 (resp.\ real operator system $V$) with nontrivial involution cannot be (resp.\ cannot have an associated) OMIN (resp.\ OMAX) operator system.   
  We remind the reader that an OMIN (resp.\ OMAX) operator system has the property that unital positive maps into (resp.\ out of) it, are ucp.

We  define the maximal quantization  ${\rm Max}(K)$  by the functorial correspondence between compact nc convex sets and operator systems, via
$A({\rm Max}(K)) = {\rm OMAX}(A(K))$.  We have that Max$(K)$ is the nc set $(K^n)$ with $K^n$  in $M_n(E^*)$ the set of unital positive (selfadjoint) linear 
 maps $\varphi : A_{\bF}(K) \to M_n$.   (Cf.\  e.g.\ \cite[End of Section 5]{EW}, where the maximal  one is defined by duality or by the bipolar theorem as the prepolar of $({\rm Min}(K))^\circ$.  For an appropriate choice of $E$ this 
 will coincide with ours because both are the largest compact nc convex set agreeing with $K$ at `level 1'.) This  is nc convex and nc compact, indeed by  a basic property of OMAX, 
  $(K^n)$ is the nc matrix state space of ${\rm OMAX}(A(K))$. That is,   

\begin{lemma} \label{omax} Let $K$ be a classical compact convex set (in a dual Banach space). Then ${\rm Max}(K) =  {\rm ncS(OMAX}(A(K)))$.
    \end{lemma}

\begin{prop} \label{omax2} Let $K$ be a classical compact convex set (in a real dual Banach space). Then 
  ${\rm Max}_{\bR}(K)_c = {\rm Max}_{\bC}(K)$ as nc convex sets, with both equaling $K$ at the first level.
    \end{prop}
            
            \begin{proof}     This is almost identical to the OMIN case.
            Complexifying the relation $A_{\bR}({\rm Max}_{\bR}(K)) =   {\rm OMAX}_{\bR}(A(K)),$  we have
$$A_{\bC}(({\rm Max}_{\bR}(K))_c) =
{\rm OMAX}_{\bR}(A(K))_c = 
{\rm OMAX}_{\bC}(A_{\bC}(K)) .$$
In the last equality we used \cite[Proposition 9.19]{BR} with $V = {\rm OMAX}(A_{\bR}(K))$, so that $V_c = A_{\bC}(K)$. 
However ${\rm OMAX}_{\bC}(A_{\bC}(K)) = A_{\bC}({\rm Max}_{\bC}(K))$ by the discussion above the lemma. 
It follows by the functorial correspondence between compact nc convex sets and operator systems that 
 $({\rm Max}_{\bR}(K))_c =  {\rm Max}_{\bC}(K)$.  \end{proof} 

 {\bf Remark.}  We saw in the proof above that $({\rm OMAX}_{\bR}(A(K)))_c = {\rm OMAX}_{\bC}(A_{\bC}(K))$.  This is related to the fact that 
 the `function system complexification' of a 
 real function system is unique.

  \bigskip
  
  For a complex operator system $V$  we  get  `minimal and maximal nc convex sets'   
    Min$(K)$ and Max$(K)$
     as exactly the nc convex sets corresponding to OMIN$(V)$ and OMAX$(V)$.     In the real case this      never works, as we have said repeatedly, unless
     $V$ has trivial (i.e.\ identity) involution. 
  
 \begin{example} \label{ointi} 
  As an illustration of Lemma \ref{omin2} 
 and Proposition \ref{omax}, 
 consider the real compact operator interval  $\bI$ from Examples \ref{ConvexInterval1} and \ref{ConvexInterval2}.
As we said there, the complexification of the real operator interval is the complex compact operator interval $\bJ$. 
If $K$ is the classical interval then $K = \bI_1 = \bJ_1$. 
It is probably well known that the complex operator system corresponding to $\bJ$ is the two dimensional function system and 
$C^*$-algebra $\ell^\infty_2(\bC)$, and similarly the real operator system corresponding to $\bI$ is
$\ell^\infty_2(\bR)$.   Indeed the reader can find a proof that works in both cases in \cite[Example 3.5]{BMcII}. 
  The latter operator system has trivial involution.    It is also well known that the usual ($C^*$-algebraic) operator system structure on
  $\ell^\infty_2(\bC)$ is both an OMAX and an OMIN.  Such equality is very rare.  Thus  
  Max$_{\bC}(K) = {\rm Min}_{\bC}(K)$ (see \cite[Example 16.4.4(1)]{Dav}).   Similarly $\bI = {\rm Max}_{\bR}(K) = {\rm Min}_{\bR}(K)$. 
  Thus indeed $$\bJ = {\rm Min}_{\bR}(K)_c = {\rm Min}_{\bC}(K) = {\rm Max}_{\bR}(K)_c = {\rm Max}_{\bC}(K)$$ as nc convex sets, with all equaling $K$ at the first level.
 \end{example}

        \begin{example} \label{newbee}    Consider the real unitary matrix $W = c(i) = \begin{bmatrix}
            0 & -1 \\ 1 & 0 
        \end{bmatrix},$ and the 
               noncommutative convex set of real matrices 
        \[K = \ncconv\{ W,-W \} = \ncconv\{ W \} \subset \cM(\bR).\]
        Somewhat surprisingly, and in stark contrast to the operator interval which as we just saw is a minimal nc convex set, 
        $K$ is not a minimal nc convex set, because the first level does not generate the entire nc convex set $K$.
However   its complexification is   a  minimal nc convex set.  
         From Theorem \ref{ncrealaffinecoml} we deduce that $K_c = \ncconv_{\bC}\{\ncconv_\bR\{ W \} \} = \ncconv_\bC\{ W \}$. A compression by $u$ shows $\pm i \in K_c$ so that $\ncconv_\bC\{\pm i\} \subseteq K_c$.  Also,
         $$W  =\frac{1}{2} \begin{bmatrix}
            1 \\i
        \end{bmatrix} i \begin{bmatrix}
            1 & -i
        \end{bmatrix} + \frac{1}{2}\begin{bmatrix}
            1 \\-i
        \end{bmatrix} (-i) \begin{bmatrix}
            1 & i
        \end{bmatrix}, $$
        so that $K_c = \ncconv_
        \bC\{\pm i\}$.  It is easy to see (multiplying by $i$) that this is nc affine isomorphic to 
        the complex operator interval         $\bJ$, so is  a  minimal nc convex set.          Since $\bJ$ is compact, so is $K_c$, and thus $K$ is a compact nc convex set by Theorem  \ref{ncrealaffinecoml}.
        We work more with these nc convex sets in the sequel \cite{BMcII}, where for instance we determine their corresponding operator systems and extreme points.  The real operator system corresponding to $K$ is the two dimensional real $C^*$-algebra $\bC_r$.  
        From this it also follows that $K$ is not a maximal nc convex set, since  the operator system corresponding to a maximal nc convex set has trivial involution. 
          \end{example}

One might think that it is obvious that  ${\rm Min}(K) \subseteq {\rm Max}(K)$ since ${\rm Max}(K)$ is a nc convex set containing $(K, \emptyset, \emptyset, \cdots)$.  However recall that ${\rm Min}(K)$ naturally `lives in'  ${\rm Max}(A(K)^*) = ({\rm Min}(A(K))^*$, being the noncommutative states on OMIN$(A(K))$, while ${\rm Max}(K)$ corresponds to the set of 
noncommutative states on OMAX$(A(K))$, and 
     naturally `lives in'  Min$(A(K)^*) = ({\rm Max}(A(K))^*$.  Of course 
     ${\rm Max}(K)$ may be viewed as consisting of the unital selfadjoint positive maps from OMIN$(A(K))$ into $M_n$, while ${\rm Min}(K)$ is the subset of ucp maps from OMIN$(A(K))$  into $M_n$.  This is not a problem for finite $n$, since one may just use the `product topology' (i.e.\ work entry-wise, see e.g.\ 1.6.4 in \cite{BLM}); the matrix spaces are isomorphic.  For infinite $n$ this suggests that perhaps the `ambient LCTVS operator space' for both ${\rm Min}(K)$ and ${\rm Max}(K)$ is Min$(A(K)^*)$.  In any case, 
     the identity map, which is the adjoint of the canonical ucp map $\Phi: {\rm OMAX}(A(K)) \to {\rm OMIN}(A(K))$,  yields a canonical nc affine continuous map 
     $\epsilon : {\rm Min}(K) \to {\rm Max}(K)$.
     It is surjective at the first  level of course.

\begin{prop} \label{omaxemb} Let $K$ be a classical compact convex set (in a dual Banach space). The canonical nc affine embedding 
  $\epsilon : {\rm Min}(K) \to {\rm Max}(K)$ is a nc topological affine embedding.  That is, $\epsilon$ is a homeomorphism onto its (compact) range for all $n$.   Similarly for a closed nc convex set $L$ in a  dual operator space $E$
                with $L_1=K$, with $L$ symmetric if $\bF = \bR$, there are   canonical nc topological affine embeddings           Min$(K) \subseteq L \subseteq {\rm Max}(K)$.   
These maps are the identity map on $K$ at level 1.     \end{prop}
            
            \begin{proof}  To see that   $\epsilon$ is one-to-one note that if $\varphi : {\rm OMIN}(A(K)) \to M_n$ satisfies  $\varphi  \circ \Phi =0$ for $\Phi$ as above, then $\varphi =0$. The statement then follows from Lemma \ref{Affineemb}. 
            
            For the last assertion we suppose that $L$ corresponds to an operator system $V$. 
            If $L$ is symmetric then $V$ has trivial  involution.  `Level 1' of $V$ is a function system with state space $K$.             The canonical ucp maps
            $${\rm OMAX}(A(K)) \to V = A(L) \to {\rm OMIN}(A(K))$$  
            dualize to give  continuous nc affine maps $${\rm Min}(K) \to L \to {\rm Max}(K) \subseteq {\rm Min}(A(K)^*).$$ 
            These maps are the identity map on $K$ at level 1. 
            As in the last paragraph these maps are one-to-one, and   nc topological affine embeddings by Lemma \ref{Affineemb}.  \end{proof} 
            
            The following is the real version of \cite[Theorem 2.5.8]{DK}.  We define $A(K_2,M_2(\bR))$ to be the (classical) affine continuous maps $K_2 \to M_2(\bR)$.  (In the result below we can also assume if desired that  these functions satisfy  the compatibility conditions (2), (3) in the definition of $A(K)$ but with all integers $\leq 2$.)

   \begin{thm}      \label{DK2.5.8}   If $K$ is a symmetric real compact nc convex set
   then the canonical restriction map $\rho : A(K) \to A(K_1)$ is an 
   isometric 
   unital order isomorphism.     More generally, if $K$ is any real compact nc convex set then the canonical restriction map 
   $\rho_2 : A(K) \to A(K_2,M_2(\bR))$ is a contractive selfadjoint unital order embedding, and $\| f \| \leq 2 \| f_{|{K_2}} \|$
   for $f \in A(K)$.     \end{thm}
            
            \begin{proof}  Clearly $\rho$ is  contractive, 
   unital and positive.   Indeed $\rho$ is  simply Kadison's function representation  (see e.g.\ Section 4.3 of 
\cite{KRI}).    
   For the second assertion, certainly $\rho_2 : A(K) \to A(K_2,M_2(\bR))$ is contractive selfadjoint unital and positive.
   Let $V = A(K), v \in V$, and recall that we may take $K_2 = {\rm UCP}(V, M_2(\bR))$. 
   Suppose that $\Phi(v) \geq 0$ (resp.\  $\| \Phi(v) \| \leq 1$) for all ucp $\Phi : V \to M_2(\bR)$.  For a complex state $\varphi$ on $V_c$ we have $c \circ \varphi_{|V}$ is a ucp $ V \to M_2(\bR)$.   Thus 
   $c(\varphi(v)) \geq 0$  (resp.\  $\| c(\varphi(v)) \| \leq 1$), and so 
   $\varphi(v) \geq 0$  (resp.\  $| \varphi(v) | \leq 1$).    By \cite[Theorem 2.5.8]{DK} 
   we have $v \geq 0$  (resp.\ $\| v \| \leq 2$). 
   \end{proof} 

  {\bf Remarks.}    1)\ The first assertions of \cite[Theorem 2.5.8]{DK} are true for all real symmetric compact nc convex sets by a similar proof (e.g.\ using the real form of the polarization identity).   
  
  \smallskip
  
  2)\ The necessity of the `symmetric' condition here, and the use of $K_2$ versus $K_1$, 
   is clear by examining the case of the quaternions.  Matt Kennedy asked us
    if one could characterize
    exactly when $\rho$ is an order isomorphism.  In fact this holds exactly when $K$ is symmetric by Corollary \ref{coti}. 
    Indeed suppose that $K$ is not symmetric, so that $V = A(K)$ has nontrivial involution, or equivalently possesses 
    nonzero antisymmetric elements.  If $0 \neq x \in V_{\rm as}$  then $\rho(x)(\varphi) = x(\varphi)$ for $\varphi \in K$, while 
    $\rho(x^*)(\varphi) = x(\varphi)$.  Since $x^*=-x$ we have $x(\varphi) = 0$.   That is, $\rho(x) = 0$. 
    Hence $\rho$ is not even bijective.

     \medskip

     The following consequence, extracted from the last proof,  may be viewed as an `improvement' on part of \cite[Corollary 3.2]{BR}.   
     The main idea in the proof we gave for that though in \cite[Corollary 3.2]{BR} is essentially the same as our proof above.  
     
\begin{cor} \label{posc} Let $V$ be a 
real operator system, and $v \in V$.  If $\Phi(v) \geq 0$  (resp.\  $\| \Phi(v) \| \leq 1$) for all ucp $\Phi : V \to M_2(\bR)$ then  $v \geq 0$  (resp.\ $\| v \| \leq 2$). 
 \end{cor}

        Similarly one may describe the noncommutative state spaces Min$_k(K)$ and Max$_k(K)$ of ${\rm OMIN}_k(V)$ and ${\rm OMAX}_k(V)$ for a real or complex operator system $V = A(K)$, and $k \in \bN$. E.g.\ we define Max$_k(K)$ by the functorial correspondence between compact nc convex sets and
operator systems, via $A({\rm Max}_k(K)) = {\rm OMAX}_k(A(K))$. So 
         ${\rm Max}_k(K) =  {\rm ncS(OMAX}_k(A(K)))$.  This agrees with $K$ up to level $k$, since $n$-positive states into $M_n$ are ucp, hence are in $K_n$ for $n \leq k$. 
         We will not however take the time to add the details here. 

\bigskip

{\bf Remarks.}  1)\ 
One cannot however expect  analogues of  Propositions \ref{omin2}  and  \ref{omax2} to hold in general for Min$_k$ and Max$_k$. 
Indeed if  $K$ is a compact nc convex set (in a real dual operator space) then 
 often ${\rm Max}_{\bR,k}(K)_c \neq {\rm Max}_{\bC,k}(K_c)$.
  Indeed this fails in general, as do the matching operator system equalities (matching via the functorial correspondence between compact nc convex sets and
operator systems).   For $k > 1$ it fails because of the problems with complexifying entanglement breaking maps as seen in \cite{CDPR2} and  \cite[Section 9]{BR} (this is spelled out in more detail in later revisions of \cite{BR}). 
For $k = 1$ it can fail because of the existence of real entangled states
  that are complex separable (i.e.\ nonentangled) (see \cite{CDPR2} and  \cite[Section 9]{BR}); and of course
  OMAX$(A(K))$ may not exist as we have said.   Indeed for the quaternions ${\rm Max}_{\bR}(K)$ has one point, while 
  $(K_c)_1$ consists of $2 \times 2$ complex density matrices.
  
  \medskip
              
        2)\   The minimal and maximal nc convex envelopes $\hat{L}$ and $\check{L}$ of a weak* closed nc convex set in a dual operator space
        $E$, constructed at the end of Section \ref{SecBipolarThrm} in connection with  Effros and Winkler's application of  the bipolar theorem, are not in full generality necessarily well related to Min and Max as we have define them above.   We now explain what we mean by this, together with an example. 
        Indeed, if $L$ is a closed nc convex set in $E$
                with $L_1=K$, then one cannot in general  expect 
        Min$(K) \subseteq L \subseteq {\rm Max}(K)$.   Here `$\subseteq$' indicates a canonical nc topological affine embedding.  
                For instance, the second  inclusion here is  certainly not valid if $L$ is the nc state space of the quaternions, where $K$ is singleton. 
        In this case ${\rm Max}(K)$ is affine isomorphic to the noncommutative state 
        space of $\bR$, a nc convex set with singleton matrix levels, while $L_2$ is not singleton.  It is interesting to compute $\check{L}$ in this example: 
        we leave it to the reader to check that its $n$th level         is  the set of unital maps $\varphi : \bH \to M_n$ with $\varphi(b)$ skew (antisymmetric) for $b \in \{ i, j, k \}$.
        That Min$(K) \subseteq L$ is usually clear in examples, but may not be very helpful (as we saw in the example at the end of Section \ref{SecBipolarThrm}).  
        
        On the other hand if $L$ is a closed nc convex set  in a dual operator space $E$, with
                 $L$ symmetric in the real case), 
                then we showed in Proposition \ref{omaxemb} that there are  topological nc affine embeddings Min$(K) \to L \to {\rm Max}(K)$.  The range of the
        first will be $\hat{L}$ since the nc convex hull of $K$ is dense in both ${\rm Min}(K)$ and $\hat{L}$.   So we may identify
        ${\rm Min}(K)$ and $\hat{L}$.   
        Alternatively, this also follows from the Remark after Proposition \ref{omin}. 
         We imagine that  there is also a way to identify ${\rm Max}(K)$ and $\check{L}$ in this case, but will not pursue this further here.  
  \end{section}

        \begin{section}{Noncommutative Functions}\label{SecNoncFunct}

        \begin{subsection}{Real NC Functions}
            Let $K$ be a real compact nc convex set. A nc function is a map $f: K \to \cM(\bR)$ that is graded, preserves direct sums, and is unitarily equivariant. More specifically, it satisfies the following properties
            \begin{enumerate}
                \item $f(K_n) \subseteq M_n(\bR)$
                \item $f(\sum \alpha_i x_i \alpha^\tran_i) = \sum \alpha_i f(x_i) \alpha^\tran_i$ for every family of $\{x_i \in K_{n_i}\}$ and collection of isometries $\{\alpha_i \in M_{n,n_i}(\bR)\}$ such that $\sum \alpha_i \alpha_i^\tran = 1_n$.
                \item $f(\beta x\beta^\tran) = \beta f(x) \beta^\tran$ for every $x \in K_n$ and unitary $\beta \in M_n(\bR)$.
            \end{enumerate}
            (Note that (3) is in fact a special case of (2).) 
            We say that $f$ is bounded if it is uniformly bounded for all $k \in K$. The space of all real bounded nc functions on $K$ is $B(K)$. This has the uniform norm
            \[||f|| = \sup_{k \in K} ||f(k)|| .\]
            We can similarly define $B(K,L)$ for $K,L$ real nc convex sets to be the nc functions from $K$ to $L$ which are bounded. Here, a nc function is the same definition as above, but with $\cM(\bR)$ replaced by $L$.
            
            As in \cite{DK}, we have
            
            \begin{lemma}\label{bkisc} 
              If $K$ is a real compact nc convex set then $B(K)$ is a real $C^*$-algebra with the uniform norm and point-wise adjoint/product.
            \end{lemma} 

            \begin{proof}
                The main difficulty in showing $B(K)$ is a $C^*$-algebra is showing it is complete.  For the reader's convenience we give the argument to show that it works in the real case.  Let $f^r \in B(K)$ be a sequence such that $\sum_{r=1}^\infty ||f^r|| < \infty$. Define $f: K \to \cM(\bF)$ by $f(x) = \sum_{r=1}^\infty f^r(x)$ for $x \in K_m$.  This converges (absolutely) since $M_n(\bF)$ is complete.  Because the $f^r$ are all graded and unitarily equivariant, $f$ will be too.  Condition (3) in the definition above is easy to check.  For (2), let $\{x_i \in K_{n_i}\}$ be set of elements in $K$, and $\alpha_i \in M_{n,n_i}(\bF)$ a family of isometries such that $\sum_i \alpha_i \alpha_i^* = 1_n$. Take $\xi, \eta \in \ell_n^2$. Then, we have that
                \begin{align*}
                    \sum_{r=1}^\infty \sum_i |\ip{\alpha_i f^r(x_i) \alpha_i^* \xi}{\eta}|
                    &\leq \sum_{n=1}^\infty \sum_i |\ip{\alpha_i f^r(x_i) \alpha_i^* \xi}{\eta}|
                    \\&\leq \sum_{n=1}^\infty \big(\sum_i ||f^r(x_i) \alpha_i^*\xi)||^2\big)^{1/2} \big(\sum_{i}||\alpha_i^*\eta||^2 \big)^{1/2}
                    \\&\leq \sum_{n=1}^{\infty}||f^r || \big(\sum_i || \alpha_i^* \xi)||^2\big)^{1/2} \big(\sum_{i}||\alpha_i^* \eta||^2 \big)^{1/2}
                    \\& = \sum_{n=1}^{\infty}||f^r || \| \xi \| \| \eta \|,
                \end{align*}
                since  e.g.\ 
                $\big(\sum_i || \alpha_i^*\xi||^2\big)^{1/2}
                    = \big(\sum_i \ip{\alpha_i \alpha_i^* \xi}{\xi}\big)^{1/2}
                    = ||\xi||.$
                Thus we may interchange the order of summation in 
                $\sum_{r=1}^\infty \sum_i \, \ip{\alpha_i f^r(x_i) \alpha_i^* \xi}{\eta}$.  We see that 
                 $$\langle f(\sum_i \, \alpha_i x_i \alpha^*_i) \xi , \eta \rangle  = \langle \sum \alpha_i f(x_i) \alpha^*_i  \xi , \eta \rangle $$ 
                as desired so that (2) holds.
                
                We check that $1 + f^* f$ is invertible  for all $f \in B(K)$.  Indeed $g(x) = (1 + f(x)^* f(x))^{-1}$ clearly defines a bounded graded function, and  checking  item (3) in the definition of nc function is easy. 
           As for item (2) in that definition, suppose that we have $\alpha_i \in M_{n,n_i}(\bR)$ such that $\sum \alpha_i \alpha_i^\tran = 1_n$ and $x_i \in K_{n_i}$. Because $p_i = \alpha_i \alpha_i^\tran$ are mutually orthogonal projections which sum to $1$ we have  $\alpha_i^\tran \alpha_j = \delta_{ij} 1_{n_i}$. So, 
            \begin{align*}
                g(\sum \alpha_i x_i \alpha_i^\tran)
                &= (1+ f(\sum \alpha_i x_i \alpha_i^\tran)^* f(\sum \alpha_j x_j \alpha_j^\tran))^{-1}
                \\&= (1+ \sum \alpha_if(x_i)^*f(x_i) \alpha_i^\tran)^{-1}
                \\&= (\sum \alpha_i (1+f(x_i)^*f(x_i))\alpha_i^\tran)^{-1}
            \end{align*}
            where the second equality is because of orthogonality. 
                       The inverse of $\sum \alpha_i (1+f(x_i)^*f(x_i))\alpha_i^\tran$ is  $\sum \alpha_i (1+f(x_i)^*f(x_i))^{-1}\alpha_i^\tran$. Indeed $(\sum p_i z_i p_i)^{-1} = \sum p_i w_i p_i$ if $z_i w_i = w_i z_i = p_i$ and $z_i, w_i \in p_i M_n p_i$.
            This completes the proof.
            \end{proof}
            As in the complex case, $A(K) \hookrightarrow B(K)$ and we define $C(K)$ to be the $C^*$-algebra generated by $A(K)$ in $B(K)$.

            \begin{subsection}{Maximal C*-algebra}
            Let $S$ be a real (or complex) operator system. The maximal $C^*$-algebra generated by $S$, denoted $C^*_{\mathrm{max}}(S)$ is the $C^*$-algebra satisfying the following universal property:

            \[
            \begin{tikzcd}
            S \arrow[r, "\iota", hookrightarrow] \arrow[dr, "\varphi"', hookrightarrow] & C^*_{\max}(S) \arrow[d, "\pi", dashed] \\
            & A
            \end{tikzcd}
            \]
            where $\iota$ is a ucoi into a $C^*$-algebra $A$ such that $C^*(\iota(S)) = C^*_{\mathrm{max}}(S)$, $\varphi$ is a ucoe such that $C^*(\varphi(S)) = A$, and $\pi$ is an induced $*$-homomorphism. 

            \begin{lemma}\label{MaxCCommC} \cite[Lemma 5.1] {BR} 
               For a real operator system   $S$ we have $C^*_{\mathrm{max}}(S_c) \cong C^*_{\mathrm{max}}(S)_c$ where the prior is the complex maximal $C^*$-algebra of the complex operator system $S_c$ and the latter is complexification of the real maximal $C^*$-algebra of $S$.
            \end{lemma}

            For any compact nc convex set $K$ we have $K \cong \ncS(A(K))$, and therefore any $k \in K_n$ corresponds to a nc state from $A(K)$ to $M_n$. The universal property of $C^*_{\text{max}}(A(K))$ gives a $*$-homomorphism $\delta_x: C^*_{\rm{max}}(A(K)) \to M_n$ such that $\delta_x \circ \iota = \hat{x}$. Taking the double adjoint gives a normal $*$-homomorphism $\delta_x^{**}: C^*_{\rm{max}}(A(K))^{**} \to M_n$ and with this we define the map $\sigma: C^*_{\rm{max}}(A(K))^{**} \to B(K)$ by 
            \[\sigma(b)(x) = \delta_x^{**}(b)\]
             for $b \in C^*_{\rm{max}}(A(K))^{**} $ and $x \in K$. Theorem $4.3.3$ of \cite{DK}
             (Theorem 16.8.10 of \cite{Dav}) shows that for a complex compact nc convex set, $B(K)$ is von Neumann algebraically  isomorphic to $C^*_{\max, \bC}(A(K))^{**}$ via the map $\sigma$.
                          Their proof also shows that $\sigma$ restricts to an isomorphism between $C^*_{\max, \bC}(A(K))$ and $C(K)$ and that elements of $C(K)$ are the point-strong continuous nc functions on $K$. 
     We will prove the real analogue of Theorem $4.3.3$ using complexification. Therefore, we need the following lemmas:

        \begin{lemma}\label{bext} For real  nc convex sets $K, L$, every real  bounded nc  map $f : K \to L$ has a unique complex  bounded nc  extension $f_c : K_c \to L_c$. 
            If $L$ is complex bounded nc affine there is a  complex  bounded nc  extension $K_c \to L$. These extensions are strongly continuous if $f$ is strongly  continuous.
            \end{lemma}

            \begin{proof} For $x+iy \in (K_c)_n$ define 
            \[f_c(x + iy) = \frac{1}{\sqrt{2}} \begin{bmatrix}
                    1_n & i \cdot 1_n
                \end{bmatrix}
                \, f(c(x+iy)) \, 
                \frac{1}{\sqrt{2}} \begin{bmatrix}
                1_n \\ -i \cdot 1_n
            \end{bmatrix}\] 
            This function is a complex nc function by a proof similar to Theorem \ref{AffineCommC}. First, it is clearly graded. Then, for $\beta \in M_n(\bC)$ a unitary and $x+iy \in (K_c)_n$ we have
            \begin{align*}
                f_c(\beta(x+iy)\beta^*) &= \frac{1}{2} \begin{bmatrix}
                    1_n & i \cdot 1_n
                \end{bmatrix}
                \, f(c(\beta)c(x+iy)c(\beta)^{\tran}) \, 
            \begin{bmatrix}
                1_n \\ -i \cdot 1_n
            \end{bmatrix} 
            \\&= \frac{1}{2} \begin{bmatrix}
                    1_n & i \cdot 1_n
                \end{bmatrix}
                \, c(\beta)f(c(x+iy))c(\beta)^{\tran} \, 
            \begin{bmatrix}
                1_n \\ -i \cdot 1_n
            \end{bmatrix}
            \\&= \beta \frac{1}{2} \begin{bmatrix}
                    1_n & i \cdot 1_n
                \end{bmatrix}
                \, f(c(x+iy)) \, \begin{bmatrix}
                1_n \\ -i \cdot 1_n
            \end{bmatrix} \beta^{*}
            \\&= \beta f_c(x+iy)\beta^{*}
            \end{align*}
            which shows $(3)$. A similar proof shows $(2)$. In the case that $f$ is bounded, then $f_c$ will be bounded by Ruan's first axiom. If $f$ is SOT continuous then $f_c$ will be too because the function $c$ is SOT continuous and so is adjoining the matrix $[1_n; -i \cdot 1_n]$.

            This will be the unique extension because for any complex nc function on $K_c$ extending $f$, say $g$, we have that
            \[(x+iy) \oplus (x-iy) = \frac{1}{2} \begin{bmatrix}
                1_n & i \cdot 1_n \\
                i \cdot 1_n & 1_n
            \end{bmatrix}c(x+iy) \begin{bmatrix}
                1_n & -i \cdot 1_n \\ -i \cdot 1_n & 1_n
            \end{bmatrix}\]
            so
            \begin{align*}
                g(x+iy) \oplus g(x-iy)
            &=g\Big(\frac{1}{2} \begin{bmatrix}
                1_n & i \cdot 1_n \\
                i \cdot 1_n & 1_n
            \end{bmatrix}c(x+iy) \begin{bmatrix}
                1_n & -i \cdot 1_n \\ -i \cdot 1_n & 1_n
            \end{bmatrix}\Big)
            \\&= \frac{1}{2} \begin{bmatrix}
                1_n & i \cdot 1_n \\
                i \cdot 1_n & 1_n
            \end{bmatrix}f(c(x+iy)) \begin{bmatrix}
                1_n & -i \cdot 1_n \\ -i \cdot 1_n & 1_n
            \end{bmatrix} .
            \end{align*}
            Here we first used that $g(x \oplus y) = g(x) \oplus g(y)$. The last part of the above equation is a $2n \times 2n$ matrix with top-right and bottom-left corners being $0$. 
                        Comparing the top-left corners of the matrices in the above equation gives
            \[g(x+iy) = \frac{1}{\sqrt{2}} \begin{bmatrix}
                    1_n & i \cdot 1_n
                \end{bmatrix}
                \, f(c(x+iy)) \, 
                \frac{1}{\sqrt{2}} \begin{bmatrix}
                1_n \\ -i \cdot 1_n 
            \end{bmatrix}\]
          as desired.   \end{proof} 
           
      \begin{lemma}\label{BoundedCommC} We have $B(K)$ is a real $W^*$-algebra, and 
             $B(K_c) \cong B(K)_c$ as complex von Neumann algebras.
            \end{lemma}
            \begin{proof}
                We use the same maps as in Theorem \ref{AffineCommC}, where we have a map
                \[\Gamma: B(K_c) \to B(K)_c\]
                \[\omega \mapsto \re \circ \omega_{|K} + i \, \im \circ \omega_{|K} \]
                with inverse $\Psi$. To see that $\Gamma$ is well-defined note
                \[||\re(\omega(x+i0))|| \leq ||\omega(x)|| \leq ||\omega|| \qquad x \in K , \]
                and similarly for the imaginary part. Therefore, 
                \[||c(\re \omega + i \im \omega)|| \leq ||\re \omega|| + ||\im \omega|| \leq 2 ||\omega|| . \]
                A similar proof to that of  Lemma \ref{reisAffine} shows that ${\rm Re}$ and $\im$ are unitarily equivariant. Also  $\Gamma$ is a unital complete order isomorphism as may be shown similarly to the proof of Theorem \ref{AffineCommC},
                 and therefore is a $*$-isomorphism. 
                 It follows that $B(K)_c$ is a $W^*$-algebra, and $\Gamma$ is automatically normal.
                 The canonical period 2 real $*$-automorphism $\theta$ on $B(K)_c$ is weak* continuous, and so 
                 its fixed point algebra, hence $B(K)$, is a real $W^*$-algebra. Therefore $B(K_c) \cong B(K)_c$ via $\Gamma$. Note that the proof  in Theorem \ref{AffineCommC} 
                 that the inverse of $\Gamma$ is $\Psi$ does not quite work here because in showing $\Psi \circ \Gamma = \Id$ we used compressions.
                 However we can adjust that proof to use unitaries as we did in the proof of Lemma \ref{bext} to make it work.  
                 Namely, for $\omega \in B(K_c)$ and $x+iy \in (K_c)_n$ we have
                 $\omega(x+iy) =   u_n^*\omega( c(x+iy))  u_n$ exactly as in the proof in the end of Lemma \ref{bext}.   It follows that 
                \begin{align*}
                    \omega(x+iy)
                &=   u_n^*\omega( c(x+iy))  u_n
                \\&= u_n^* [(\re \circ \omega)(c(x+iy))+ i (\im \circ \omega)(c(x+iy))] u_n
                \\&= \Psi((\re \circ \omega)+ i (\im \circ \omega))(x+iy) = \Psi(\Gamma (\omega))(x+iy).
                \end{align*} 
                 That $\Gamma \circ \Psi = \Id$ is exactly as in the proof in the end of Theorem \ref{AffineCommC}.    
            \end{proof}
                \end{subsection}

        \begin{prop}\label{CCommC}
            $C(K_c) \cong C(K)_c$ as complex $C^*$-algebras.
        \end{prop}
        \begin{proof} We have $$C(K_c) 
                = C^*(A(K_c)) \cong C^*(A(K)_c)
                \cong C^*(A(K))_c = C(K)_c.$$
               We used the fact that 
               if $A$ is a subsystem 
               of a real $C^*$-algebra $B$ then 
               $C^*(A_c) = C^*(A) + i C^*(A)$ in $B_c$.
        \end{proof}

        \begin{lemma}\label{PStrongCommC}
            Let $S(K)$ be the point-strong continuous functions in $B(K)$. Then $S(K_c) \cong S(K)_c$ as $C^*$-algebras coming from the congruence $B(K_c) \cong B(K)_c$.
        \end{lemma}
        \begin{proof}
            We just need that the map $\Gamma$ in \ref{BoundedCommC} satisfies
            $\Gamma(S(K_c)) =  S(K)_c$. Let $\omega \in S(K_c)$ and $x_\lambda \xrightarrow[]{}x \in K_n$.
            Because $\re$ is a contraction $\re(\omega(x_\lambda+i0)) \to \re(\omega(x+i0))$  in the strong operator topology. Similarly for $\text{Im}$. Therefore, 
            \[c(\re \omega + i \im \omega)(x_\lambda) \xrightarrow[]{\text{SOT}}c(\re \omega + i \im \omega)(x) .\]
            Thus $\Gamma(\omega) \in S(K)_c$.  The converse is similar using the map $\Psi$ defined above Theorem \ref{AffineCommC}. 
        \end{proof}

The following is the real case of \cite[Theorem 4.3.3]{DK}.

        \begin{theorem}      Let $K$ be a real compact nc convex set. Then the map $\sigma: C^*_{\rm{max}}(A(K))^{**} \to B(K)$ is a real linear normal $*$-isomorphism, which restricts to a $*$-isomorphism from $C^*_{\rm{max}}(A(K))$ onto $C(K)$. The elements of $C(K)$ are the point-strong  continuous nc functions on $K$. Also, $\sigma \circ \iota$ is the identity map on $A(K)$.
               \end{theorem}

        \begin{proof}
            For $\cA$ a real $C^*$-algebra, we have $(\cA_c)^{*} \cong (\cA^{*})_c$ by \cite{RComp}. 
            The ensuing map 
            $(\cA_c)^{**} \to (\cA^{**})_c$ is a unital 
            complete order isomorphism and normal $*$-isomorphism of complex von Neumann algebras \cite{Li}.        Moreover, 
            \begin{align*}
                (C^*_{\rm{max},\bR}(A(K))^{**})_c
                &\cong (C^*_{\rm{max},\bR}(A(K))_c)^{**}
                \\&\cong C^*_{\rm{max},\bC}(A(K)_c)^{**}
                \\&\cong C^*_{\rm{max},\bC}(A(K_c))^{**}
            \end{align*} 
 Using this, Lemma \ref{MaxCCommC} and  Theorem \ref{AffineCommC} give
            The complex case of this theorem and Lemma \ref{BoundedCommC} gives
            \[C^*_{\rm{max},\bC}(A(K_c))^{**} \cong B(K_c) \cong B(K)_c , \] with the composition 
            $C^*_{\rm{max},\bC}(A(K))^{**} \cong B(K)_c \cong B(K_c)$ being via a (complex) normal $*$-isomorphism $\pi$ say. 
            There is a real embedding of $C^*_{\rm{max},\bR}(A(K))^{**}$ into $(C^*_{\rm{max},\bR}(A(K))^{**})_c$, and similarly for $B(K)$ into $B(K)_c$. A diagram chase shows that the restriction of the complex normal $*$-isomorphism above is a real  normal $*$-isomorphism  $C^*_{\rm{max},\bR}(A(K))^{**} \cong B(K)$, which is the `identity map' on  the copies of $A(K)$.  From this we see again that $B(K)$ is a von Neumann algebra.  
            
            To do the diagram chase, take $f \in A(K)$. This  embeds into $C^*_{\rm{max}, \bR}(A(K))^{**}$ and is denoted by $\theta(f)$. Going `up' in the diagram gives an element in $(C^*_{\rm{max}, \bR}(A(K))^{**})_c$, namely $\theta(f)+i0$. The isomorphism $(C^*_{\rm{max},\bR}(A(K))^{**})_c \cong C^*_{\rm{max},\bC}(A(K)_c)^{**}$ will take $\theta(f)+i0$ to $\theta(f+i0)$. The third congruence in the above centered equations takes our element to $\theta(\psi(f+i0))$. Taking $\sigma$ of this element and evaluating at $x+i0 \in K_c$ for $x \in K_n$ will give an element of $M_n(\bC)^{**}$. We evaluate this functional at $A \in M_n(\bC)^*$ and get
            \begin{align*}
            \sigma(\theta(\psi(f+i0)))(x+i0)(A)
                &= \delta_{x+i0}^{**}(\theta(\psi(f+i0)))(A)
                \\&= A(\delta_{x+i0}(\psi(f+i0)))
                \\&= A(\psi(f+i0)(x+i0))
                \\&= A(f(x))
            \end{align*}
            On the other hand, 
            \[\sigma(\theta(f))(x)(A) = \delta_x^{**}(\theta(f))(A) = A(f(x))\]
            Therefore, the diagram chase shows that it is the
            `identity map' on  the copies of $A(K)$ that 
            extends to $\pi$. 
            Since $\pi(A(K))$ is the copy of $A(K)$  in $B(K)$ inside $B(K_c)$, it 
            follows that $\pi(C^*_{\rm{max}}(A(K))^{**})$ is a $C^*$-subalgebra $D$ of $B(K)$ with $D + i D = B(K)_c$.
            Hence $D = B(K) = \pi(C^*_{\rm{max}}(A(K))^{**})$.
            Also $I_{A(K)}$ extends to a $*$-isomorphism between the $C^*$-algebra generated by $A(K)$ in both sets, so
            $C^*_{\rm{max}}(A(K)) \cong C(K)$. It also shows that $\sigma \circ \iota$ is the identity map on $A(K)$. 
                        That $C(K)$ are the point-strong  continuous nc functions follows from the complex case and Lemma \ref{PStrongCommC}.  Indeed  $C(K) = B(K) \cap C(K)_c = B(K) \cap C(K_c)$.
         \end{proof}

         As we saw after Lemma \ref{MaxCCommC}, any element $k$ of $K$ defines a nc/matrix state on $C(K)$, and    a    weak* continuous 
         matrix state on $B(K)$.
         In particular, $f \mapsto f(k)$ is weak* continuous on $B(K)$.

        \begin{cor}
            Let $K$ be a real compact nc convex set. The real enveloping von Neumann algebra $C(K)^{**}$ of $C(K)$ is $*$-isomorphic to the real von Neumann  algebra $B(K)$ of bounded nc functions on $K$. 
        The real dual operator system $A(K)^{**}$ is completely order isomorphic to 
                the real operator system of bounded nc affine functions on $K$.  The latter space has as
        complexification the bounded complex nc affine functions on $A_c$.
        \end{cor}

      \begin{proof}  To see that $A(K)^{**}$ is completely order isomorphic to the real operator system of bounded nc affine functions on $K$, note that $A(K)^{\perp \perp} \subseteq B(K)$. We need to show $f \in B(K)$ is in 
        $A(K)^{\perp \perp}$ if and only if $f(\beta^* x\beta) = \beta^* f(x) \beta$ for every $x \in K_m$ and isometry $\beta \in M_{m,n}(\bR)$. Since this the latter is true for $f \in A(K)$, it will also be true by 
        a weak* approximation argument 
        for $f \in A(K)^{\perp \perp}$,
        using the fact above the corollary.  This gives a weak* continuous ucoe   $\nu : A(K)^{\perp \perp} \to BA(K)$, where  $BA(K)$ is the operator subsystem of bounded nc affine functions in $B(K)$.  Then
        $BA(K)_c = BA(K) + i BA(K) = BA_{\bC}(K_c)$ in $B(K)_c = B(K_c)$.  
        The rest follows by complexification from         Theorem    4.3.3 in \cite{DK}
        (Theorem 16.8.10 of \cite{Dav}).  Indeed if $\nu$ were not surjective then its complexification would also not be, 
        contradicting that $BA_{\bC}(K_c) \cong A_{\bC}(K_c)^{\perp \perp}$. \end{proof}
       
               \begin{prop}
            Let $K$ be a real compact nc convex set and $f \in C(K)$ a continuous nc function. Then $f$ is bounded with 
            \[||f|| = \sup_{n < \infty} ||f|_{K_n}|| . \]
        \end{prop}
        
            In  \cite[Proposition 2.5.3]{DK} (or \cite[Remark 16.6.2 (2)]{Dav})   there is a more general version of the last result
    in the complex case;   this holds with the same  proof  in the real case.

        \end{subsection}

        \begin{subsection}{Minimal $C^*$-algebra}
            As in the complex case, every real operator system $V$ has a $C^*$-envelope or minimal $C^*$-algebra denoted by $C^*_{\min}(V)$. There is a ucoe $\iota: V \to C^*_{\min}(V)$ which satisfies the following universal property:
        
        \[
        \begin{tikzcd}
        & A = C^*(\varphi(V)) \arrow[d, two heads, "\pi"] \\
        V \arrow[ur, "\varphi"] \arrow[r, "\iota"'] & C^*_{\min}(V)
        \end{tikzcd}
        \]
        where $\varphi$ is a ucoe of $V$ into another real $C^*$-algebra $A$ such that $\varphi(V)$ generates $A$ as a $C^*$-algebra, and $\pi$ is an induced surjective $*$-homomomorphism making the diagram commute.

        \begin{lemma}
            \cite[Corollary 4.3]{BCK} For a real operator system $S$ we have $C^*_{\min}(S_c) \cong C^*_{\min}(S)_c$ where the prior is the complex minimal $C^*$-algebra of the complex operator system $S_c$ and the latter is complexification of the real minimal $C^*$-algebra of $S$.
        \end{lemma}

        \begin{example}
            If $\cA$ is a real/complex $C^*$-algebra viewed as an operator system, the universal property shows $C^*_{\min}(\cA) = \cA$. There exists $C^*$-algebras that are not the complexification of real $C^*$-algebras by \cite[Problem 1.5]{Ros} and we can use this to construct complex operator systems and compact nc convex sets which are not complexifications. For instance, let $A$ be such a complex $C^*$-algebra viewed as an operator system, then if it were the complexification of a real operator system $V$ we would get
            \[\cA = C^*_{\min}(\cA) = C^*_{\min}(V_c) \cong C^*_{\min}(V)_c\]
            and the latter is the complexification of a real $C^*$-algebra which is a contradiction. Similarly, if every complex compact nc convex set was a complexification, then $\ncS(\cA)$ would be $L_c$ for some real compact nc convex set $L$. By Theorem \ref{AffineCateg} we have that $\cA \cong A(L_c)$ as complex operator systems and so
            \[\cA \cong C^*_{\min}(A(L_c)) \cong C^*_{\min}(A(L))_c , \]
            which is a contradiction.
        \end{example}
        \end{subsection}

        \end{section}

  \subsection*{Acknowledgements}  
  We acknowledge support from NSF Grant DMS-2154903.  We thank Travis Russell for several conversations and helpful thoughts and information, and thank Matt Kennedy 
  for discussion and who very recently, in answer to a couple of questions of ours related to our draft,  kindly drew our attention   to some items in the ArXiV revision of \cite{DK} of June 10, 2025.  Similarly we thank Scott McCullough for very kind input (solicited by us) on our paper, and on real matrix convexity, and for drawing our attention 
  to several recent papers in that area.   He mentioned to us for example 
  that there is some  complexification procedure 
  in the finite dimensional setting of \cite[Subsection 4.1]{Ev}. 
  We also thank James Pascoe for some comments
  at a similar time, and Ken Davidson for drawing our attention to \cite{Dav}.  We also thank the referee for identifying some things that were not clearly expressed in our preprint, and for several comments and useful suggestions. 
    Finally we mention two related recent papers by the first author released in January 2026: ``Base norm spaces--classical, complex, and noncommutative" (with D. M. Hay), and  
``Regularity of compact convex sets--classical and noncommutative''.   We saw in the present paper that a real nc compact convex set corresponds to a real nc state space.   However this is 
 essentially also the real {\em nc base} for a real {\em nc base norm space}, and similarly in the complex case.

\end{document}